\documentclass{scrartcl}
\usepackage{latexsym} 
\usepackage{amsmath}
\usepackage{amsthm}
\usepackage{amsfonts}
\usepackage{amssymb}
\usepackage{amsxtra}
\usepackage{amscd}
\usepackage[shortlabels]{enumitem}
\usepackage{mathrsfs}
\usepackage{hyperref}
\usepackage{mathtools}
\usepackage{bbm}

\hypersetup{
breaklinks=true,
colorlinks=true,
linkcolor=blue,
citecolor=blue,
urlcolor=blue,
filecolor=blue,
}

\numberwithin{equation}{section} 
\setkomafont{title}{\normalfont}

\newenvironment{pdeq}{ \left\{ \begin{aligned}}{\end{aligned}\right.}

\newcommand{\eqrefsub}[2]{\ensuremath{\eqref{#1}_{#2}}}
\newcommand{\np}[1]{(#1)}
\newcommand{\nb}[1]{[#1]}
\newcommand{\bp}[1]{\big(#1\big)}
\newcommand{\bb}[1]{\big[#1\big]}
\newcommand{\Bp}[1]{\bigg(#1\bigg)}

\newcommand{\Bb}[1]{\bigg[#1\bigg]}
\newcommand{\calf}{{\mathcal F}}

\newcommand{\caln}{{\mathcal N}}

\newcommand{\calp}{{\mathcal P}}

\newcommand{\cals}{{\mathcal S}}
\newcommand{\calt}{{\mathcal T}}

\newcommand{\R}{\mathbb{R}}
\newcommand{\Q}{\mathbb{Q}}
\newcommand{\C}{\mathbb{C}}
\newcommand{\Z}{\mathbb{Z}}
\newcommand{\N}{\mathbb{N}}
\DeclareMathOperator{\e}{e}

\DeclareMathOperator{\Div}{div}

\DeclareMathOperator{\supp}{supp}

\DeclareMathOperator{\dist}{dist}

\DeclareMathOperator{\rot}{rot}

\newcommand{\embeds}{\hookrightarrow}

\newcommand{\ra}{\rightarrow}

\newcommand{\wto}{\rightharpoonup}
\newcommand{\set}[1]{\ensuremath{\{#1\}}}

\newcommand{\setc}[2]{\ensuremath{\{#1 : #2\}}}
\newcommand{\setcl}[2]{\ensuremath{\bigl\{#1 : #2\bigr\}}}
\newcommand{\setcL}[2]{\ensuremath{\biggl\{#1\, :\, #2\biggr\}}}

\newcommand{\ball}{B}

\renewcommand{\restriction}[2]{#1\big | _{#2}}
\newcommand{\proj}{\calp}

\newcommand{\projhm}{\proj_{\Omega_m}}
\newcommand{\grp}{G}
\newcommand{\dualgrp}{\widehat{G}}

\newcommand{\grph}{H}

\newcommand{\torus}{{\mathbb T}}

\newcommand{\transpose}{\top}

\newcommand{\idmatrix}{I}
\newcommand{\Rn}{{\R^n}}

\newcommand{\ddt}{\frac{{\mathrm d}}{{\mathrm d}t}}

\newcommand{\grad}{\nabla}

\newcommand{\pt}{\partial_t}

\newcommand{\dx}{{\mathrm d}x}

\newcommand{\dt}{{\mathrm d}t}

\newcommand{\dxi}{{\mathrm d}\xi}

\newcommand{\dtheta}{{\mathrm d}\theta}

\newcommand{\SR}{\mathscr{S}}

\newcommand{\TDR}{\mathscr{S^\prime}}

\newcommand{\ft}[1]{\widehat{#1}}

\newcommand{\FT}{\mathscr{F}}
\newcommand{\iFT}{\mathscr{F}^{-1}}

\newcommand{\AR}{\mathrm{A}}
\newcommand{\norm}[1]{\lVert#1\rVert}

\newcommand{\norml}[1]{\bigl\lVert#1\bigr\rVert}

\newcommand{\snorm}[1]{{\lvert #1 \rvert}}
\newcommand{\snorml}[1]{{\bigl\lvert #1 \big\rvert}}
\newcommand{\snormL}[1]{{\Bigl\lvert #1 \Big\rvert}}

\newcommand{\WSR}[2]{\mathrm{W}^{#1,#2}} 
\newcommand{\WSRN}[2]{\mathrm{W}^{#1,#2}_0}

\newcommand{\WSRloc}[2]{\mathrm{W}^{#1,#2}_{\mathrm{loc}}} 
\newcommand{\CR}[1]{\mathrm{C}^{#1}}  
\newcommand{\LR}[1]{\mathrm{L}^{#1}}
 
\newcommand{\lR}[1]{\ell^{#1}}

\newcommand{\LRloc}[1]{\mathrm{L}^{#1}_{\mathrm{loc}}} 
\newcommand{\CRi}{\CR \infty}
\newcommand{\CRci}{\CR \infty_0}

\newcommand{\LRsigma}[1]{\mathrm{L}^{#1}_{\sigma}}

\newcommand{\XreyT}{{\mathcal X_{\rey,\tay}^q}}
\newcommand{\Xreys}{{\mathrm X_{\rey,\tay,s}^q}}
\newcommand{\Xreyk}{{\mathrm X_{\rey,\tay,\perf k}^q}}
\newcommand{\XzeroT}{\mathcal X_{0,\tay}^q}

\newcommand{\YT}{\mathcal Y^q}
\newcommand{\Ys}{\mathrm Y^q}

\newcommand{\nsnonlinb}[2]{#1\cdot\grad #2}

\newcommand{\vvel}{v}
\newcommand{\vpres}{p}

\newcommand{\Vvel}{V}
\newcommand{\Vpres}{P}

\newcommand{\wvel}{w}
\newcommand{\wpres}{\mathfrak{q}}

\newcommand{\twvel}{\widetilde{w}}

\newcommand{\twpres}{\widetilde{\pi}}

\newcommand{\uvel}{u}

\newcommand{\upres}{\mathfrak{p}}

\newcommand{\Uvel}{U}

\newcommand{\tuvel}{\widetilde{u}}
\newcommand{\tupres}{\widetilde{\mathfrak{p}}}

\newcommand{\zvel}{z}

\newcommand{\rotterm}[1]{\np{ \eone\wedge #1 - \nsnonlinb{\eone\wedge x}{#1}}}

\newcommand{\rottermsimple}[1]{ \eone\wedge #1 - \nsnonlinb{\eone\wedge x}{#1} }

\newcommand{\tin}{\text{in }}
\newcommand{\tif}{\text{if }}
\newcommand{\ton}{\text{on }}

\newcommand{\tfor}{\text{for }}

\newcommand{\half}{\frac{1}{2}}

\renewcommand{\epsilon}{\varepsilon}
\renewcommand{\phi}{\varphi}
\newcommand{\rey}{\lambda}

\newcommand{\tay}{\omega}

\newcommand{\per}{\calt}
\newcommand{\iper}{\frac{1}{\per}}
\newcommand{\perf}{\frac{2\pi}{\per}}
\newcommand{\perfs}{\tfrac{2\pi}{\per}}

\newcommand{\eone}{\e_1}

\newcommand{\rotmatrix}{Q}

\newcommand{\charfct}{\mathbbm{1}}
\newcommand{\cutoff}{\chi}

\newcommand{\change}[1]{}

\renewcommand{\eqrefsub}[2]{\eqref{#1}\textsubscript{#2}}

\theoremstyle{plain}
\newtheorem{thm}{Theorem}[section]

\newtheorem{lem}[thm]{Lemma}
\newtheorem{prop}[thm]{Proposition}
\newtheorem{cor}[thm]{Corollary}
\theoremstyle{remark}
\newtheorem{rem}[thm]{Remark}

\begin{document}
\title{On the Oseen-type resolvent problem associated with
time-periodic flow past a rotating body}

\author{Thomas Eiter%
\footnote{%
Weierstrass Institute for Applied Analysis and Stochastics,
Mohrenstra\ss{}e 39, 10117 Berlin, Germany.
Email: {\texttt{thomas.eiter@wias-berlin.de}}} 
}

\maketitle

\begin{abstract}
Consider the time-periodic flow of an incompressible viscous fluid past a body performing a rigid motion
with non-zero translational and rotational velocity.
We introduce a framework of homogeneous Sobolev spaces
that renders the resolvent problem of the associated linear problem 
well posed on the whole imaginary axis.
In contrast to the cases without translation or rotation, 
the resolvent estimates are merely uniform 
under additional restrictions,
and the existence of time-periodic solutions
depends on the ratio of the rotational velocity of the body motion
to the angular velocity associated with the time period.
Provided that this ratio is a rational number,
time-periodic solutions to both the linear
and, under suitable smallness conditions, the nonlinear problem
can be established.
If this ratio is irrational,
a counterexample shows that in a special case 
there is no uniform resolvent estimate
and solutions to the time-periodic linear problem do not exist.
\end{abstract}


\section{Introduction}

Consider the system
\begin{equation}\label{sys:Oseen.rot.res}
\begin{pdeq}
is\vvel+\tay\rotterm{\vvel} 
- \Delta \vvel
-\rey\partial_1\vvel
+ \grad \vpres
 &= g
&& \tin \Omega, \\
\Div\vvel&=0
&& \tin\Omega, \\
\vvel&=0
&& \ton \partial\Omega,
\end{pdeq}
\end{equation}
where $\Omega\subset\R^3$ is a three-dimensional exterior domain,
and $\rey,\,\tay>0$ and $s\in\R$ are given parameters. 
The function $g\colon\Omega\to\R^3$ is a given vector field,
and the unknown solution $\np{\vvel,\vpres}$ consists of
the vector field $\vvel\colon\Omega\to\R^3$ and the scalar field $\vpres\colon\Omega\to\R$.
Problem \eqref{sys:Oseen.rot.res}
can be regarded as a resolvent problem with 
a purely imaginary resolvent parameter $is$, $s\in\R$.
In this article we provide function classes 
where the existence of a unique solution to \eqref{sys:Oseen.rot.res}
can be established.
We further investigate the availability of uniform resolvent estimates,
which allow to establish solutions to the associated 
time-periodic linear problem
\begin{equation}\label{sys:Oseen.rot.tp}
\begin{pdeq}
\pt\uvel+\tay\rotterm{\uvel} 
- \Delta \uvel
-\rey\partial_1\uvel
+ \grad \upres
 &= f
&& \tin \torus\times\Omega, \\
\Div\uvel&=0
&& \tin \torus\times\Omega, \\
\uvel&=0
&& \ton \torus\times\partial\Omega.
\end{pdeq}
\end{equation}
This linear theory can then be applied to 
show existence of solutions to the nonlinear problem
\begin{equation}\label{sys:NS.rot.tp}
\begin{pdeq}
\partial_t \uvel + \tay\rotterm{\uvel} 
-\rey\partial_1\uvel+ \uvel\cdot\grad\uvel
&= f + \Delta \uvel - \grad \upres 
&& \tin \torus\times\Omega, \\
\Div\uvel&=0
&& \tin \torus\times\Omega, \\
\uvel&=\rey\eone+\tay\eone\wedge x
&& \ton \torus\times\partial\Omega, \\
\lim_{\snorm{x}\to\infty} \uvel(t,x) &= 0
&& \tfor t\in \R,
\end{pdeq}
\end{equation}
which describes the time-periodic flow of 
an incompressible viscous fluid past a rigid body 
that moves 
with (non-vanishing, time-independent) translational and rotational velocities
$\rey\eone$ and $\tay\eone$ for $\rey,\,\tay>0$.
Here $\Omega\subset\R^3$ is the exterior domain surrounding the rigid body,
and to indicate that all occurring functions are time periodic,
the time axis is given by the torus group 
$\torus\coloneqq\R/\per\Z$ for a prescribed time period $\per>0$.
The functions $\uvel\colon\torus\times\Omega\to\R^3$
and $\upres\colon\torus\times\Omega\to\R$ 
denote velocity and pressure fields of the fluid,
expressed in a frame attached to the body, 
and $f\colon\torus\times\Omega\to\R^3$ is an external body force.
Here, viscosity and density constants are set equal to $1$,
and the fluid is assumed to be attached to the boundary of the body.
Moreover, \eqrefsub{sys:NS.rot.tp}{4} 
indicates that the flow is at rest at infinity.
Since this condition will later be included
in the definition of the function spaces 
in a generalized sense,
we omitted
the corresponding equation in \eqref{sys:Oseen.rot.res} and \eqref{sys:Oseen.rot.tp}.
Our analysis of the time-periodic problems will be based on the 
study of the resolvent problem \eqref{sys:Oseen.rot.res}.
Note that if $\uvel$ is a $\per$-periodic solution to \eqref{sys:Oseen.rot.tp},
then the Fourier coefficient of order $k\in\Z$ satisfies \eqref{sys:Oseen.rot.res}
with $s=\perf k$.
This also explains why we only consider purely imaginary resolvent parameters $is$, $s\in\R$,
in this article.

The analysis of solutions to the nonlinear time-periodic problem \eqref{sys:NS.rot.tp}
was initiated by Galdi and Silvestre \cite{GaldiSilvestre_ExistenceTPSolutionsNSAroundMovingBody_2006},
who derived the existence of weak solutions
when the body performs a general time-periodic rigid motion.
However, the established $\LR{2}$ framework was not appropriate to capture the 
spatial asymptotic properties of the flow.
This issue was addressed in a recent article by Galdi \cite{Galdi2020_ExistenceUniquenessAsBehRegularTPViscFlowAroundMovingBodyRotCase},
who showed the existence of regular solutions
subject to pointwise decay estimates.
An alternative approach that reflects the 
asymptotic behavior away from the body
is based on the fundamental work by Yamazaki \cite{Yamazaki2000},
who considered the time-periodic flow around a body at rest,
that is, system \eqref{sys:NS.rot.tp} for $\rey=\tay=0$.
He established solutions in $\LR{3,\infty}(\Omega)$,
also known as weak-$\LR{3}(\Omega)$, 
by exploiting well-known $\LR{p}$-$\LR{q}$ smoothing estimates 
for the Stokes semigroup.
Similar estimates
were derived by Shibata \cite{Shibata_OseenSemigroupRotatingEffect_2008}
for the semigroup in the case $\rey,\,\tay>0$.
Using these estimates, 
Yamazaki's method 
leads to time-periodic solutions to \eqref{sys:NS.rot.tp}
in $\LR{3,\infty}(\Omega)$,
as was later shown by Geissert, Hieber and Nguyen \cite{GeissertHieberNguyen_TP2016},
who developed a general approach to time-periodic problems 
based on semigroup theory.
Recently, Eiter and Kyed 
\cite{EiterKyed_ViscousFlowAroundRigidBodyPerformingTPMotion_2021}
used a different method
based on a direct analysis of the linear problem
\eqref{sys:Oseen.rot.tp}
without relying on the associated semigroup,
and existence of solutions to \eqref{sys:NS.rot.tp} was shown
such that the velocity field belongs to 
$\LR{q}(\Omega)$, $q\in(2,\infty)$,
under the assumption that the rotational velocity $\tay$ and the time period $\per$ 
are related by $\tay=\perf$.
This severe restriction
already appears in the existence theory for
the linear problem \eqref{sys:Oseen.rot.tp}
derived in \cite{EiterKyed_ViscousFlowAroundRigidBodyPerformingTPMotion_2021}.
In contrast, 
in the cases without translation ($\rey=0$)
or without rotation ($\tay=0$),
existence of time-periodic solutions to \eqref{sys:Oseen.rot.tp}
can be shown without 
further restrictions on the time-period $\per$;
see \cite{Eiter2021_StokesResTPFlowRotating,GaldiKyed_TPflowViscLiquidpBody_2018}.
A leading question of this article is whether the condition
$\tay=\perf$
is necessary for the existence of 
time-periodic solutions to \eqref{sys:Oseen.rot.tp} and \eqref{sys:NS.rot.tp}
if $\rey,\,\tay>0$,
and in how far it can be weakened.

Since the additional term $\tay\rotterm{\uvel}$ for $\tay>0$
may be regarded as a differential operator with unbounded coefficient,
the linear problem \eqref{sys:Oseen.rot.tp}
cannot be treated as a lower-order perturbation of the case $\tay=0$.
Instead, 
we shall handle this term by a method 
recently developed by Galdi and Kyed \cite{KyedGaldi_asplqesoserfI,KyedGaldi_asplqesoserfII}
to investigate steady-state solutions to \eqref{sys:Oseen.rot.tp},
that is,
solutions to \eqref{sys:Oseen.rot.res} for $s=0$.
This method was successfully applied to the time-periodic problem \eqref{sys:Oseen.rot.tp} 
for $\rey=0$ in \cite{Eiter2021_StokesResTPFlowRotating}
and for $\rey>0$ in
\cite{EiterKyed_ViscousFlowAroundRigidBodyPerformingTPMotion_2021}.
As mentioned above, while the linear theory in \cite{Eiter2021_StokesResTPFlowRotating}
holds for all $\per,\,\tay>0$,
in \cite{EiterKyed_ViscousFlowAroundRigidBodyPerformingTPMotion_2021}
the assumption $\perf=\tay$ was imposed.
This restriction makes it possible to absorb the term 
$\tay\rotterm{\uvel}$ into the time derivative $\pt\uvel$
by a suitable transformation when $\Omega=\R^3$,
and to reduce the problem to the case $\tay=0$.
Note that if $\tay\neq\perf$,
then the associated transformation is not an isomorphism between $\per$-periodic functions.
To circumvent this problem, we proceed as in \cite{Eiter2021_StokesResTPFlowRotating},
and adapt this method 
to first derive well-posedness of the resolvent problem \eqref{sys:Oseen.rot.res}
by a reduction to the auxiliary problem
\begin{equation}\label{sys:Oseen.tp.mod.R3.intro}
\begin{pdeq}
is\uvel
+\pt\uvel
- \Delta \uvel
-\rey\partial_1\uvel
+ \grad \upres 
&= f
&&\tin\torus\times\R^3,
\\
\Div\uvel
&=0 
&& \tin\torus\times\R^3.
\end{pdeq}
\end{equation}
This system may be regarded as the mixture of the classical Oseen resolvent problem 
and the time-periodic Oseen problem.
Since the relevant differential operator in \eqref{sys:Oseen.tp.mod.R3.intro}
has constant coefficients,
a solution formula can directly be deduced in terms of a Fourier multiplier in 
the group setting $\torus\times\R^3$,
and we can derive associated \textit{a priori} estimates by $\LR{q}$ multiplier theorems,
which lead to resolvent estimates for \eqref{sys:Oseen.rot.res}
that are uniform for all $s\in\R$
satisfying $\dist(s,\tay\Z\setminus\set{s})>\delta$ for some fixed $\delta>0$.
This result differs from the case $\rey=0$,
where uniform resolvent estimates for all $s\in\R$ are available
(see \cite{Eiter2021_StokesResTPFlowRotating}).
This observation parallels the known results for the non-rotating case $\tay=0$, that is,
for the resolvent problem 
\begin{equation}
\label{sys:Oseen.res}
\begin{pdeq}
is\vvel
- \Delta \vvel
-\rey\partial_1\vvel
+ \grad \vpres &= g
&& \tin \Omega, \\
\Div\vvel&=0
&& \tin \Omega, \\
\vvel&=0
&& \ton \partial\Omega.
\end{pdeq}
\end{equation}
In the Stokes case ($\rey=0$),
there exists a constant $C>0$ 
such that for all 
$s\in\R\setminus\set{0}$
and $g\in\LR{q}(\Omega)^3$, $q\in(1,\infty)$, 
the velocity field $\vvel$ of the (unique) solution $\np{\vvel,\vpres}$ 
to \eqref{sys:Oseen.res} satisfies
\begin{equation}\label{est:Oseen.res.intro}
\snorm{s}\norm{\vvel}_{q}\leq C\norm{g}_{q}.
\end{equation}
In contrast, in the Oseen case ($\rey>0$),
an analogous statement can only be shown 
with a uniform constant $C$ as long as $\snorm{s}\geq \delta$
for some $\delta>0$.
Moreover, one cannot expect the validity of
a uniform estimate as $s\to0$,
as was pointed out by 
Deuring and Varnhorn \cite{DeuringVarnhorn_OseenResolventEst_2010},
who constructed a counterexample
in the special case $q=2$ and $\Omega=\R^3$.
Our results below show that 
in the rotating case $\tay>0$ the situation becomes even more involved,
and we cannot derive a uniform estimate if $s$ approaches $\tay\Z$.
Similarly to \cite{DeuringVarnhorn_OseenResolventEst_2010},
we further construct a counterexample
showing that uniform resolvent estimates cannot exist in the case $q=2$ and $\Omega=\R^3$.

The described phenomenon is in accordance with the following observation:
For $\tay\geq0$ 
we can understand \eqref{sys:Oseen.rot.res}
as the resolvent problem $(is-A_\tay)\vvel=g$
of a closed operator $A_\tay\colon D(A_\tay)\to\LRsigma{q}(\Omega)$
with domain $D(A_\tay)\subset\LRsigma{q}(\Omega)$,
where $\LRsigma{q}(\Omega)$ is the class of all solenoidal functions in $\LR{q}(\Omega)^3$.
Then the (essential) spectra of $A_\tay$ and $A_0$ are related by
\[
\sigma_\text{ess}(A_\tay)=\sigma_\text{ess}(A_0)+i\tay\Z;
\]
see \cite{FarwigNeustupa_SpectralPropertiesLq_2010}.
Therefore, one would expect that the singular behavior of problem \eqref{sys:Oseen.res} 
at $s=0$ 
can be observed in a similar fashion for solutions to \eqref{sys:Oseen.rot.res}
at any $s\in\tay\Z$.
Indeed, this is what our findings indicate.
Moreover, since $0\in\sigma_\text{ess}(A_0)$, 
problem \eqref{sys:Oseen.rot.res}
is ill-posed for all $s\in\tay\Z$ in this functional framework of closed operators.
Therefore, we introduce a different framework
and show that \eqref{sys:Oseen.rot.res} can be rendered well-posed within homogeneous Sobolev spaces.
In particular, we shall not derive an estimate of the form \eqref{est:Oseen.res.intro},
which would contradict $i\tay\Z\subset\sigma_\text{ess}(A_\tay)$,
but the non-standard resolvent estimate
\begin{equation}
\label{est:Oseen.rot.res.intro}
\norm{is\vvel+\tay\rotterm{\vvel}}_{q}
\leq C \norm{g}_{q}.
\end{equation}
Clearly, in the case $\tay=0$, 
estimate \eqref{est:Oseen.rot.res.intro} reduces to \eqref{est:Oseen.res.intro}. 

Based on the analysis of the resolvent problem \eqref{sys:Oseen.rot.res}, 
we then investigate the time-periodic linear problem \eqref{sys:Oseen.rot.tp}.
Provided that we have a uniform resolvent estimate for the relevant resolvent parameters,
that is, that $\dist(s,\tay\Z\setminus\set{s})>\delta$ for all $s\in\perf\Z$,
we then show existence of solutions to \eqref{sys:Oseen.rot.tp} 
in a framework of absolutely convergent Fourier series.
We see below that this assumption is satisfied if and only if $\perf/\tay$ is a rational number.
Under this condition and suitable smallness assumptions on the data $f$, $\rey$ and $\tay$,
we further establish existence of a solution to the nonlinear problem \eqref{sys:NS.rot.tp}.
Moreover, the aforementioned counterexample for the resolvent problem \eqref{sys:Oseen.rot.res}
enables us to derive a result on the non-existence 
of a time-periodic solution to the linear problem \eqref{sys:Oseen.rot.tp}
in $\LR{2}(\R^3)$
if $\perf/\tay\not\in\Q$.
This suggests that the restriction to the case $\perf/\tay\in\Q$
may also be necessary for existence in the general setting $\LR{q}(\Omega)$.

We first introduce the basic notation in Section \ref{sec:notation},
which allows us to formulate the main results of this article in Section \ref{sec:mainresults}.
In Section \ref{sec:preliminaries} we prepare some preliminary results.
Section \ref{sec:resprob.wholespace} and Section \ref{sec:resprob.extdom} focus on the well-posedness of 
the resolvent problem \eqref{sys:Oseen.rot.res} in the whole space $\Omega=\R^3$
and in an exterior domain $\Omega\subset\R^3$, respectively.
In Section \ref{sec:time-periodic} we prove the existence results
for the time-periodic problems \eqref{sys:Oseen.rot.tp} and \eqref{sys:NS.rot.tp}.
Finally, in Section \ref{sec:counterexample} we construct the counterexample 
to the uniformity of the resolvent estimate \eqref{est:Oseen.res.intro}
and conclude the non-existence result
for the linear time-periodic problem \eqref{sys:Oseen.rot.tp}.

\section{Notation}
\label{sec:notation}

For a fixed period $\per>0$,
the symbol $\torus\coloneqq\R/\per\Z$ denotes the 
associated torus group. 
Occasionally, we identify elements of $\torus$ 
with their unique representatives in $[0,\per)$.
Points $(t,x)\in\torus\times\R^3$, 
consist of a time variable $t\in\torus$ and a space variable $x=(x_1,x_2,x_3)\in\R^3$.
We write $\snorm{x}$ for the Euclidean norm of $x$,
and
$x\cdot y$, $x\wedge y$ and $x\otimes y$
denote the scalar, vector and tensor products of $x,y\in\R^3$,
We further use the notation $x\wedge y\cdot z\coloneqq\np{x\wedge y}\cdot z$ for $x,y,z\in\R^3$.

For derivatives in time and space we write 
$\pt$ and $\partial_j\coloneqq\partial_{x_j}$, $j=1,2,3$, respectively,
and $\grad$, $\Div$ and $\Delta$ denote (spatial) gradient, divergence and Laplace operator. 
By $\grad^2\uvel$ we denote the collection of all second-order spatial derivatives of a function $\uvel$.

We use the symbol $C$
to denote a generic positive constant 
that may change from line to line. 
When we want to emphasize that $C$ depends on a specific set of quantities 
$\set{a,b,\dots}$, 
we write $C=C(a,b,\dots)$.

Unless stated otherwise, 
$\Omega\subset\R^3$ always denotes a three-dimensional 
exterior domain,
that is, $\Omega$ is a domain that is the complement of a compact nonempty set.
We let $\ball_R\subset\R^3$ denote the ball
of radius $R>0$ centered at $0$,
and we set $\Omega_R\coloneqq\Omega\cap\ball_R$.

Classical Lebesgue and Sobolev spaces are denoted by $\LR{q}(\Omega)$
and $\WSR{k}{q}(\Omega)$ for $q\in[1,\infty]$ and $k\in\N$,
and we write $\norm{\cdot}_{q;\Omega}$ 
and $\norm{\cdot}_{k,q;\Omega}$ for the associated norms. 
If the underlying domain is clear from the context, 
we simply write $\norm{\cdot}_{q}$ 
and $\norm{\cdot}_{k,q}$.
The same convention is used for
the norm $\norm{\cdot}_{q,\torus\times\Omega}$
of the Lebesgue space $\LR{q}(\torus\times\Omega)$ in space and time.
We further define $\WSRN{1}{q}(\Omega)$ as the closure of $\CRci(\Omega)$ in $\WSR{1}{q}(\Omega)$,
where $\CRci(\Omega)$ is the set of all smooth functions with compact support in $\Omega$,
and $\WSR{-1}{q'}(\Omega)$ 
is the dual space of $\WSRN{1}{q}(\Omega)$,
where $1/q+1/q'=1$.
We denote the norm of $\WSR{-1}{q'}(\Omega)$ 
by $\norm{\cdot}_{-1,q'}$.
The classes $\LRloc{q}(\Omega)$ and $\WSRloc{k}{q}(\Omega)$
consist of all functions that 
locally belong to $\LR{q}(\Omega)$ and $\WSR{k}{q}(\Omega)$,
respectively.

When clear from the context,
we often do not distinguish between 
a space $X$ and its $n$-fold Cartesian product $X^n$, $n\in\N$.
Moreover, $\norm{\cdot}_X$ denotes the norm of a normed vector space $X$.
The symbol $\LR{q}(\torus;X)$ 
denotes the Bochner--Lebesgue space for $q\in[1,\infty]$,
and we define
$\WSR{1}{q}(\torus;X)\coloneqq\setcl{\uvel\in\LR{q}(\torus;X)}{\pt\uvel\in\LR{q}(\torus;X)}$.
Here the torus group $\torus$ is always equipped with
the normalized Haar measure such that
\[
\forall f\in\CR{}(\torus):\quad
\int_\torus f(t)\,\dt
\coloneqq
\frac{1}{\per}
\int_0^\per f(t')\,\dt',
\]
where $\CR{}(\torus)$ denotes the class of continuous functions on $\torus$.

In the case $\Omega=\R^3$, 
the space-time domain is given by 
the locally compact abelian group $\grp\coloneqq\torus\times\R^3$.
The dual group of $\grp$ can be identified with 
$\dualgrp\coloneqq\Z\times\R^3$.
By $\SR(\grp)$ we denote
the associated Schwartz--Bruhat space,
and $\TDR(\grp)$ is its dual space, the space of tempered distributions. 
Both were first introduced by Bruhat \cite{Bruhat61},
see also \cite{EiterKyed_tplinNS_PiFbook} for more details.
We define the Fourier transform 
$\FT_\grp$ on $\grp$
by
\[
\begin{aligned}
\FT_\grp\colon\SR(\grp)\ra\SR(\dualgrp), 
&&\FT_\grp\nb{\uvel}(k,\xi)
&\coloneqq\int_\torus\int_{\R^3} \uvel(t,x)\e^{-ix\cdot\xi-ik t}\,\dx\dt,\\
\iFT_\grp\colon\SR(\dualgrp)\ra\SR(\grp), 
&&\iFT_\grp\nb{\wvel}(t,x)
&\coloneqq\sum_{k\in\Z}\,\int_{\R^3} \wvel(k,\xi)\e^{ix\cdot\xi+ik t}\,\dxi.
\end{aligned}
\]
Then $\FT_\grp$ is an isomorphism
with inverse $\iFT_\grp$
provided that the Lebesgue measure $\dxi$ is suitably normalized.
By duality, the Fourier transform also becomes 
an isomorphism between the corresponding spaces of tempered distributions.
By analogy, we define the Fourier transform 
on the groups $\torus$ and $\R^3$ as
\[
\begin{aligned}
\FT_{\torus}\colon\SR(\torus)&\ra\SR(\Z), 
&\qquad
\FT_{\torus}\nb{\uvel}(k)
&\coloneqq\int_\torus \uvel(t)\e^{-ik t}\,\dt,\\
\iFT_{\torus}\colon\SR(\Z)&\ra\SR(\torus), 
&\qquad
\iFT_{\torus}\nb{\wvel}(t)
&\coloneqq\sum_{k\in\Z}\wvel(k)\e^{ik t},
\\
\FT_{\R^3}\colon\SR(\R^3)&\ra\SR(\R^3), 
&\qquad
\FT_{\R^3}\nb{\uvel}(\xi)
&\coloneqq\int_{\R^3} \uvel(x)\e^{-ix\cdot\xi}\,\dx,\\
\iFT_{\R^3}\colon\SR(\R^3)&\ra\SR(\R^3), 
&\qquad
\iFT_{\R^3}\nb{\wvel}(x)
&\coloneqq\int_{\R^3} \wvel(\xi)\e^{ix\cdot\xi}\,\dxi.
\end{aligned}
\]

For the investigation of the time-periodic problem \eqref{sys:Oseen.rot.tp},
we also work within the space of absolutely convergent $X$-valued Fourier series 
given by
\begin{align}\label{eq:def.Aspace}
\begin{aligned}
\AR(\torus;X)
&\coloneqq
\setcL{f\colon\torus\to X}{f(t)=\sum_{k\in\Z}f_k \e^{ikt}, \ f_k\in X, \ 
\sum_{k\in\Z}\norm{f_k}_{X}<\infty},
\\
\norm{f}_{\AR(\torus;X)}
&\coloneqq\sum_{k\in\Z}\norm{f_k}_{X}
\end{aligned}
\end{align}
for a normed space $X$.
When $X$ is a Banach space,
then $\AR(\torus;X)$ coincides with the Banach space
$\iFT_\torus\bb{\lR{1}(\Z;X)}$, 
whence $\AR(\torus;X)\embeds\CR{}(\torus;X)$.
Since useful inequalities 
can directly be transferred from spaces $X$ to the corresponding 
spaces $\AR(\torus;X)$ (see 
\cite[Prop.~3.1 and 3.2]{EiterKyed_ViscousFlowAroundRigidBodyPerformingTPMotion_2021} for example),
these spaces also provide a useful framework for the treatment of nonlinear time-periodic problems.
For simplicity, we also write $\uvel\in\AR(\torus;\WSRloc{k}{q}(\Omega))$
if $\uvel\in\AR(\torus;\WSR{k}{q}(K))$
for all compact sets $K\subset\Omega$.

Next we formulate the function spaces
for solutions to the time-periodic problems 
\eqref{sys:Oseen.rot.tp} and \eqref{sys:NS.rot.tp}.
For fixed $\per,\,\rey,\,\tay>0$ 
and $q\in(1,2)$,
we define the space for the time-periodic velocity field by
\[
\begin{aligned}
\XreyT(\torus\times\Omega)\coloneqq
\setcl{\uvel\in&\AR(\torus;\WSRloc{2}{q}(\Omega)^3)}{\\
&\grad^2\uvel, \,
\pt\uvel+\tay\rotterm{\uvel}, \,
\partial_1\uvel \in\AR(\torus;\LR{q}(\Omega)),\\
&\qquad\quad
\uvel\in\AR(\torus;\LR{2q/(2-q)}(\Omega)), \ 
\grad\uvel\in\AR(\torus;\LR{4q/(4-q)}(\Omega))}.
\end{aligned}
\]
The function class for the pressure is independent of $\rey,\,\tay>0$ and 
given by
\[
\YT(\torus\times\Omega)
\coloneqq
\setcl{\upres\in\AR(\torus;\WSRloc{1}{q}(\Omega))}{
\grad\upres\in\AR(\torus;\LR{q}(\Omega)),\
\upres\in\AR(\torus;\LR{3q/(3-q)}(\Omega))}.
\]
For the analysis of the resolvent problem \eqref{sys:Oseen.rot.res},
the function class for the velocity fields
additionally depends on $s\in\R$
and is defined by
\[
\begin{aligned}
\Xreys(\Omega)\coloneqq
\setcl{\vvel\in\WSRloc{2}{q}(\Omega)^3}{
\grad^2\vvel, \,
&is\vvel+\tay\rotterm{\vvel}, \,
\partial_1\vvel \in\LR{q}(\Omega),\\
&\qquad\ 
\vvel\in\LR{2q/(2-q)}(\Omega), \ 
\grad\vvel\in\LR{4q/(4-q)}(\Omega)},
\end{aligned}
\]
and the corresponding pressure is characterized by
\[
\Ys(\Omega)
\coloneqq
\setcl{\vpres\in\WSRloc{1}{q}(\Omega)}{
\grad\vpres\in\LR{q}(\Omega),\
\vpres\in\LR{3q/(3-q)}(\Omega)}.
\]
These spaces are constructed in such a way that 
if a $\per$-time-periodic function 
$\uvel$ belongs to $\XreyT(\torus\times\Omega)$, 
then its $k$-th Fourier coefficient 
$\uvel_k\coloneqq\FT_\torus\nb{\uvel}(k)$, $k\in\Z$,
belongs to $\Xreys(\Omega)$ for $s=\perf k$.
Similarly, Fourier coefficients of elements $\upres\in\YT(\torus\times\Omega)$ belong to $\Ys(\Omega)$.

\section{Main Results}
\label{sec:mainresults}

Here we collect our main results on the 
well-posedness of the
resolvent problem \eqref{sys:Oseen.rot.res}
and the time-periodic problems \eqref{sys:Oseen.rot.tp} and \eqref{sys:NS.rot.tp}.
For the whole section, let $\Omega=\R^3$ or $\Omega\subset\R^3$ 
be an exterior domain with $\CR{3}$-boundary.
At first, we address the resolvent problem \eqref{sys:Oseen.rot.res}.

\begin{thm}\label{thm:Oseen.rot.res}
Let $\rey>0$, $s\in\R$ and $0<\tay\leq\tay_0$,
and let $q\in(1,2)$ and $g\in\LR{q}(\Omega)^3$.
Then there exists a unique solution 
$\np{\vvel,\vpres}\in\Xreys(\Omega)\times\Ys(\Omega)$
to \eqref{sys:Oseen.rot.res},
which obeys the estimate
\begin{equation}\label{est:Oseen.rot.res}
\begin{aligned}
&\norm{\dist(s,\tay\Z) \,\vvel}_q
+\norm{is\vvel+\tay\rotterm{\vvel}}_{q}
+\norm{\grad^2 \vvel}_{q}
+\rey\norm{\partial_1\vvel}_{q}
\\
&\qquad\qquad
+\norm{\grad \vpres}_{q}
+\rey^{1/4}\norm{\grad \vvel}_{4q/(4-q)}
+\rey^{1/2}\norm{\vvel}_{2q/(2-q)}
+\norm{\vpres}_{3q/(3-q)}
\leq C\norm{g}_{q}
\end{aligned}
\end{equation}
for a constant $C=C(\Omega,q,\rey,\tay,s)>0$.
If $\theta>0$ such that
\begin{equation}\label{est:rey2.res} 
\rey^2\leq\theta\min\setcl{\snorm{s-\tay k}}{k\in\Z,\,s\neq\tay k},
\end{equation}
then $C=C(\Omega,q,\rey,\tay_0,\theta)>0$, that is, 
$C$ is independent of $s$ and $\tay$.
If additionally $q\in(1,3/2)$, then $C$ can be chosen 
uniformly in $s$, $\tay$ and $\rey$,
that is, 
such that $C=C(\Omega,q,\tay_0,\theta)>0$.
\end{thm}

Observe that
we may reformulate \eqref{est:rey2.res}
as
\[
\rey^2\leq\theta\dist(s,\tay\Z\setminus\set{s})=
\begin{cases}
\theta\tay &\tif s\in\tay\Z,
\\
\theta\dist(s,\tay\Z) &\tif s\not\in\tay\Z.
\end{cases}
\]
One readily sees that there is no $\theta>0$ such that 
this condition is satisfied for all $s\in\R$ at the same time.
Therefore,
Theorem \ref{thm:Oseen.rot.res}
does not allow to choose the same constant $C$ for all $s\in\R$.
As explained in the introduction, 
this lack of a uniform constant is not surprising 
since the same phenomenon occurs for \eqref{sys:Oseen.res}, 
the Oseen resolvent problem without rotation,
as $s\to0$.

To obtain a $\per$-periodic solution 
to \eqref{sys:Oseen.rot.tp} in terms of a Fourier series,
a uniform constant for all $s\in\perf\Z$ is necessary.
As follows from Proposition \ref{prop:linearcomb} below,
condition \eqref{est:rey2.res} can only be satisfied for all $s\in\perf\Z$
if the quotient $\perf/\tay$ is a rational number.
This observation leads to the following existence result 
for the time-periodic problem \eqref{sys:Oseen.rot.tp}.

\begin{thm}\label{thm:Oseen.rot.tp}
Let $\per,\,\rey>0$ and $0<\tay\leq\tay_0$,
and let $q\in(1,2)$ and $f\in\AR(\torus;\LR{q}(\Omega)^3)$.
If $\perf/\tay\in\Q$,
then there exists a unique $\per$-periodic solution 
$\np{\uvel,\upres}\in\XreyT(\torus\times\Omega)\times\YT(\torus\times\Omega)$
to \eqref{sys:Oseen.rot.tp},
which obeys the estimate
\begin{equation}\label{est:Oseen.rot.tp}
\begin{aligned}
&\norm{\pt\uvel+\tay\rotterm{\uvel}}_{\AR(\torus;\LR{q}(\Omega))}
+\norm{\grad^2 \uvel}_{\AR(\torus;\LR{q}(\Omega))}
+\rey\norm{\partial_1\uvel}_{\AR(\torus;\LR{q}(\Omega))}
\\
&\qquad\quad
+\norm{\grad \upres}_{\AR(\torus;\LR{q}(\Omega))}
+\rey^{1/4}\norm{\grad \uvel}_{\AR(\torus;\LR{4q/(4-q)}(\Omega))}
+\rey^{1/2}\norm{\uvel}_{\AR(\torus;\LR{2q/(2-q)}(\Omega))}
\\
&\qquad\qquad\qquad\qquad\qquad\qquad\qquad\qquad\quad\
+\norm{\upres}_{\AR(\torus;\LR{3q/(3-q)}(\Omega))}\leq C\norm{f}_{\AR(\torus;\LR{q}(\Omega))}
\end{aligned}
\end{equation}
for a constant $C=C(\Omega,q,\rey,\tay,\per)>0$.
If $\theta>0$ such that 
\begin{equation}
\label{est:rey2.tp}
\rey^2\leq\theta\min\setcl{a>0}{a\in\perf\Z+\tay\Z},
\end{equation}
then $C=C(\Omega,q,\rey,\tay_0,\theta)>0$, that is, 
$C$ is independent of $\per$ and $\tay$.
If additionally $q\in(1,3/2)$, then $C$ can be chosen 
independently of $\tay$, $\per$ and $\rey$,
that is, 
such that $C=C(\Omega,q,\tay_0,\theta)>0$.
\end{thm}

Observe that, due to the linear structure of $\perf\Z+\tay\Z$,
the condition
\eqref{est:rey2.tp} can be reformulated as
\[
\rey^2
\leq\theta\min\setcl{\snorm{a-b}}{a,b\in\perf\Z+\tay\Z,\ a\neq b}
\]
or
\[
\rey^2
\leq\theta\min\setcl{\snorm{s-\tay k}}{s\in\perf\Z,\,k\in\Z,\,s\neq\tay k}.
\]
This shows that \eqref{est:rey2.tp} is directly obtained from \eqref{est:rey2.res}
by taking the minimum over all $s\in\perf\Z$.
As follows from
Proposition \ref{prop:linearcomb} below,
existence and positivity of this minimum
are ensured by the restriction to $\perf/\tay\in\Q$.
As explained above, 
this restriction
is due to the lack of a uniform estimate 
for the resolvent problem \eqref{sys:Oseen.rot.res}
as $s$ approaches $\tay\Z$.

\begin{rem}\label{rem:est.dist}
In contrast to the other terms on the left-hand side of  \eqref{est:Oseen.rot.res},
the term $\norm{\dist(s,\tay\Z) \,\vvel}_q$ 
does not directly correspond to any of the terms in \eqref{est:Oseen.rot.tp}.
However, if we let $A_1\coloneqq\setc{k\in\Z}{\perf k\in\tay\Z}$
and $A_2\coloneqq\setc{k\in\Z}{\perf k\not\in\tay\Z}$
and decompose
the velocity field $\uvel$
as 
\[
\uvel=\uvel^{(1)}+\uvel^{(2)}, 
\qquad
\uvel^{(1)}\coloneqq\sum_{k\in A_1}\uvel_k \e^{i\perf kt},
\qquad
\uvel^{(2)}\coloneqq\sum_{k\in A_2}\uvel_k \e^{i\perf kt},
\]
then our proof below also yields the estimate
\[
\min\setcl{a>0}{a\in\perf\Z+\tay\Z} \norml{\uvel^{(2)}}_q\leq C\norm{f}_{q}.
\]
As mentioned above, the existence and positivity of the minimum is due to $\perf/\tay\in\Q$.
\end{rem}

For the treatment of the nonlinear problem \eqref{sys:NS.rot.tp},
first observe that $\rey$ and $\tay$ appear as data on the right-hand side of \eqrefsub{sys:NS.rot.tp}{3}.
Therefore, to obtain a solution to \eqref{sys:NS.rot.tp} for ``small'' data,
it is important that the constant $C$
in the \textit{a priori} estimate \eqref{est:Oseen.rot.tp}
can be controlled as $\rey,\tay\to0$.
By Theorem \ref{thm:Oseen.rot.tp},
this is the case if \eqref{est:rey2.tp} holds and $q<3/2$.
Under these conditions and suitable smallness assumptions,
we can derive existence 
of solutions to the nonlinear problem \eqref{sys:NS.rot.tp}.

\begin{thm}\label{thm:NS.rot.tp}
Let $q\in[\frac{12}{11},\frac{6}{5}]$, $\rho\in(\frac{3q-3}{q},1]$ and $\per,\theta,\kappa>0$.
Then there exists $\rey_0>0$ such that for all $\rey,\tay>0$ with $\perf/\tay\in\Q$ and
\begin{equation}\label{est:rey2.nonlin}
\rey\leq\rey_0, \qquad
\tay\leq\kappa\rey^\rho, \qquad
\rey^2\leq\theta\min\setcl{a>0}{a\in\perf\Z+\tay\Z},
\end{equation}
there exists $\varepsilon>0$
such that for all $f\in\AR(\torus;\LR{q}(\Omega)^3)$ 
with $\norm{f}_{\AR(\torus;\LR{q}(\Omega)^3)}<\varepsilon$
there is a solution $\np{\uvel,\upres}\in\XreyT(\torus\times\Omega)\times\YT(\torus\times\Omega)$
to \eqref{sys:NS.rot.tp}.
\end{thm}

For the proof we will proceed as in \cite{EiterKyed_ViscousFlowAroundRigidBodyPerformingTPMotion_2021},
where the case $\tay=\perf$ was treated,
and combine the linear theory from Theorem \ref{thm:Oseen.rot.tp}
with a fixed-point argument.
Note that by following \cite{EiterKyed_ViscousFlowAroundRigidBodyPerformingTPMotion_2021},
one could also allow for a time-dependent translational velocity,
where $\rey$ in \eqref{sys:NS.rot.tp} is replaced with
a time-periodic function $\alpha\colon\torus\to\R$
that satisfies suitable smallness conditions
and has a non-zero mean value $\rey\coloneqq\int_\torus\alpha\,\dt>0$.

As explained above, \eqref{est:rey2.nonlin} cannot be satisfied for $\rey>0$ if $\perf/\tay\not\in\Q$,
since this implies that $\inf\setcl{a>0}{a\in\perf\Z+\tay\Z}=0$.
In contrast, if $\perf/\tay\in\Q$,
one can find suitable parameters $\rey$, $\tay$ to satisfy \eqref{est:rey2.nonlin}. 
Indeed, if $\perf/\tay=c/d$ with $c,d\in\N$ coprime,
then 
\[
\min\setcl{a>0}{a\in\perf\Z+\tay\Z}
=\frac{\tay}{d}\min\setcl{a>0}{a\in c\Z+d\Z}
=\frac{\tay}{d}=\frac{2\pi}{\per c}.
\]
To satisfy \eqref{est:rey2.nonlin}
for given $\per>0$, 
one can thus fix $d\in\N$ and decide $\tay$ by choosing a number $c\in\N$ coprime to $d$ and so large that
$\rey^2/\theta\leq\perf/c=\tay/d\leq\kappa\rey^\rho/d$,
which is possible if $\rey>0$ is sufficiently small.

Now a natural question is 
what happens for $\perf/\tay\not\in\Q$ 
and whether
the exclusion of this case 
in Theorem \ref{thm:Oseen.rot.tp} and Theorem \ref{thm:NS.rot.tp}
is only a remnant of our proof 
or a necessary condition
for the existence of a $\per$-time-periodic solution.
The conjecture that it might be necessary
is supported by the following result.
It shows the existence of a right-hand side $f$ 
such that a time-periodic solution to the linear problem \eqref{sys:Oseen.rot.tp}
satisfying an estimate of the form \eqref{est:Oseen.rot.tp}
cannot exist in the case $q=2$ and $\Omega=\R^3$.
Actually, we show two non-existence results, 
one in the setting of absolutely convergent Fourier series 
of the previous theorems,
and one in the more general setting of $\LR{2}(\torus\times\R^3)$ functions.

\begin{thm}\label{thm:counterexample.tp}
Let $\Omega=\R^3$, $q=2$ and $\rey,\tay, \per>0$
such that $\perf/\tay\not\in\Q$.
\begin{enumerate}[label=\roman*.]
\item \label{it:counterexample.tp.i}
There is $f\in\AR(\torus;\LR{2}(\R^3)^3)$ 
such that there exists no time-periodic solution
$\np{\uvel,\upres}\in\LRloc{1}(\torus\times\R^3)^{3+1}$
to \eqref{sys:Oseen.rot.tp}
with
$\pt\uvel+\tay\rotterm{\uvel},\,
\grad^2\uvel,\,
\partial_1\uvel
\in\AR(\torus;\LR{2}(\R^3)^3)$.
\item \label{it:counterexample.tp.ii}
There is $f\in\LR{2}(\torus\times\R^3)^3$ 
such that there exists no time-periodic solution
$\np{\uvel,\upres}\in\LRloc{1}(\torus\times\R^3)^{3+1}$
to \eqref{sys:Oseen.rot.tp}
with
$\pt\uvel+\tay\rotterm{\uvel},\,
\grad^2\uvel,\,
\partial_1\uvel
\in\LR{2}(\torus\times\R^3)^3$.
\end{enumerate}
\end{thm}

In order to prove these non-existence results,
we construct a counterexample 
that shows that there is no uniform resolvent estimate
if $\rey>0$, $q=2$ and $\Omega=\R^3$.
More precisely, we construct a sequence of 
resolvent parameters, right-hand sides and corresponding solutions
to the resolvent problem \eqref{sys:Oseen.rot.res}
that violates the existence of a uniform constant in estimate 
\eqref{est:Oseen.rot.res.intro} with $q=2$.

\begin{thm}\label{thm:counterexample}
Let $\Omega=\R^3$, $q=2$ and $\rey,\tay>0$.
Then there exist sequences $(s_n)\subset\R$
and $(g_n)\subset\LR{2}(\R^3)^3$
and a sequence of solutions 
$(\vvel_n,\vpres_n)\subset\WSRloc{2}{2}(\R^3)^3\times\WSRloc{1}{2}(\R^3)$
to \eqref{sys:Oseen.rot.res} (with $s=s_n$ and $g=g_n$)
such that
$\grad^2\vvel_n,\ is\vvel_n+\tay\rotterm{\vvel_n},\,
\partial_1\vvel_n\in\LR{2}(\torus\times\R^3)$
and
\begin{equation}\label{est:counterexample}
\norm{is_n\vvel_n+\tay\rotterm{\vvel_n}}_{2}
\geq C n^{1/2} \norm{g_n}_{2}
\end{equation}
for a constant $C$ independent of $n$. 
Moreover, for any $\alpha>0$ such that $\alpha/\tay\not\in\Q$,
we can choose $(s_n)\subset\alpha\Z$
with $\dist(s_n,\tay\Z)\to0$ as $n\to\infty$.
\end{thm}

\section{Preliminaries}
\label{sec:preliminaries}

Before we start with the proofs of the main theorems,
we collect some auxiliary results in this section.
We begin with the following estimates,
which may be regarded as anisotropic versions of Sobolev's inequality.
For the sake of generality, we consider the $n$-dimensional case.

\begin{lem}\label{lem:embedding.oseen}
Let $n\in\N$, $n\geq2$ and $q\in(1,\infty)$.
Then there exists a constant $C=C(n,q)>0$ such that
for all $\rey>0$ and
all $\vvel\in\WSRloc{2}{q}(\Rn)$ with 
$\grad^2\vvel,\,\partial_1\vvel\in\LR{q}(\Rn)$
and $\vvel\in\LR{r}(\Rn)$ for some $r\in[1,\infty)$
it holds
\begin{align}
\label{est:embedding.oseen.grad}
\rey^{1/(n+1)}\norm{\grad\vvel}_{s_1}
&\leq C\norm{\Delta\vvel+\rey\partial_1\vvel}_q
&&\qquad\tif q\in(1,n+1),
\\
\label{est:embedding.oseen.fct}
\rey^{2/(n+1)}\norm{\vvel}_{s_2}
&\leq C\norm{\Delta\vvel+\rey\partial_1\vvel}_q
&&\qquad\tif q\in(1,\frac{n+1}{2}),
\end{align}
where 
\[
s_1\coloneqq\frac{(n+1)q}{n+1-q},
\qquad
s_2\coloneqq\frac{(n+1)q}{n+1-2q}.
\]
\end{lem}

\begin{proof}
At first, assume that $\vvel\in\SR(\Rn)$.
Then we can express $\vvel$ and $\partial_j\vvel$, $j=1,\dots,n$,
by means of the Fourier transform such that
\[
\rey^{1/(n+1)}\partial_j\vvel
=\iFT_{\Rn}\bb{m_1\,\FT_{\Rn}\nb{\Delta\vvel+\rey\partial_1\vvel}},
\quad
\rey^{2/(n+1)}\vvel
=\iFT_{\Rn}\bb{m_2\,\FT_{\Rn}\nb{\Delta\vvel+\rey\partial_1\vvel}}
\]
with
\[
m_1(\xi)
\coloneqq\frac{\rey^{1/(n+1)}i\xi_j}{i\rey\xi_1-\snorm{\xi}^2},
\qquad
m_2(\xi)
\coloneqq\frac{\rey^{2/(n+1)}}{i\rey\xi_1-\snorm{\xi}^2}.
\]
As in the proof of \cite[Lemma VII.4.2]{GaldiBookNew},
one readily deduces from Lizorkin's multiplier theorem
that $m_1$ and $m_2$ are Fourier multipliers
such that
\[
\begin{aligned}
&\forall g\in\SR(\Rn):
\quad 
&\norml{\iFT_{\Rn}\bb{m_1\,\FT_{\Rn}\nb{g}}}_{s_1}
&\leq C\norm{g}_{q} \quad \tif q<n+1,
\\
&\forall g\in\SR(\Rn):
\quad
&\norml{\iFT_{\Rn}\bb{m_2\,\FT_{\Rn}\nb{g}}}_{s_2}
&\leq C\norm{g}_{q} \quad \tif q<(n+1)/2,
\end{aligned}
\]
where $C$ is independent of $\rey$.
Together with the above representation formulas, 
this property directly implies \eqref{est:embedding.oseen.grad} 
and \eqref{est:embedding.oseen.fct}
for $\vvel\in\SR(\Rn)$.
For general $\vvel\in\LR{r}(\Rn)$ with $\grad^2\vvel,\,\partial_1\vvel\in\LR{q}(\Rn)$,
the inequalities now follow from an approximation by functions from $\SR(\Rn)$.
\end{proof}

The following elementary result
will explain how the restriction $\perf/\tay\in\Q$ 
comes into play
in the existence result of Theorem \ref{thm:Oseen.rot.tp}.

\begin{prop}
\label{prop:linearcomb}
Let $\alpha,\tay>0$ and define $M\coloneqq\setcl{k\alpha+\ell\tay}{k,\ell\in\Z}$.
Then $M$ is discrete in $\R$ if and only if $\alpha/\tay\in\Q$,
and $M$ is dense in $\R$ if and only if $\alpha/\tay\not\in\Q$.
In particular,
\begin{equation}
\label{eq:linearcomb.inf}
\inf\setcl{a>0}{a\in M}
=
\begin{cases}
\min\setcl{a>0}{a\in M}>0 
& \tif \alpha/\tay\in\Q,
\\
0 
&\tif \alpha/\tay\not\in\Q.
\end{cases}
\end{equation}
\end{prop}

\begin{proof}
It is well known that $N\coloneqq\setcl{k\beta+\ell}{k,\ell\in\Z}\subset\R$
is discrete if $\beta\in\Q$, and $N$ is dense if $\beta\not\in\Q$.
Choosing $\beta=\alpha/\tay$, we have $M=\tay N$,
and the whole statement follows directly.
\end{proof}

For the construction of the sequence of resolvent parameters 
in Theorem \ref{thm:counterexample},
we use that
we still obtain a dense set 
after restricting to 
odd multiples of $\tay$
if $\alpha/\tay\not\in\Q$.

\begin{cor}
\label{cor:linearcomb.odd}
Let $\alpha,\tay>0$ with $\alpha/\tay\not\in\Q$.
Then $\alpha\Z+\tay(2\Z+1)$ is dense in $\R$.
\end{cor}

\begin{proof}
Let $a\in\R$.
Since $(2\alpha)/(4\tay)\in\R\setminus\Q$,
Proposition \ref{prop:linearcomb}
shows existence of sequences $(b_n), (c_n)\subset 2\alpha\Z+4\tay\Z$ 
such that $b_n\to a$ and $c_n\to a-2\tay$ as $n\to\infty$.
We define $a_n\coloneqq \half(b_n+c_n+2\tay)$.
Then $(a_n)\subset\alpha\Z+2\tay\Z+\tay=\alpha\Z+\tay(2\Z+1)$
and $a_n\to a$ as $n\to\infty$.
This completes the proof.
\end{proof}

\section{The resolvent problem in the whole space}
\label{sec:resprob.wholespace}

We begin to study 
the resolvent problem \eqref{sys:Oseen.rot.res}
in the case $\Omega=\R^3$,
where it reduces to
\begin{equation}\label{sys:Oseen.rot.res.R3}
\begin{pdeq}
is\vvel
+\tay\rotterm{\vvel}
- \Delta \vvel
- \rey\partial_1\vvel
+ \grad \vpres 
&= g
&&\tin\R^3,
\\
\Div\vvel
&=0 
&& \tin\R^3.
\end{pdeq}
\end{equation}
This investigation
prepares the analysis of the exterior-domain problem \eqref{sys:Oseen.rot.res}
in the subsequent section.
Moreover, it allows us to establish solutions 
to the associated time-periodic problem
\eqref{sys:Oseen.rot.tp} for $\Omega=\R^3$.
For this purpose,
it is important to keep track of the constants
in the resolvent estimates for \eqref{sys:Oseen.rot.res.R3}.
The main result of this section reads as follows.
\begin{thm}\label{thm:Oseen.rot.res.R3}
Let $q\in(1,\infty)$,
and let $\rey>0,\,\tay>0$ and $s\in\R$.
For each $g\in\LR{q}(\R^3)^3$ there exists 
a solution 
$\np{\vvel,\vpres}\in\WSRloc{2}{q}(\R^3)^3\times\WSRloc{1}{q}(\R^3)$ 
to \eqref{sys:Oseen.rot.res.R3}
that satisfies
\begin{equation}
\begin{aligned}
\norm{\dist(s,\tay\Z)\,\vvel}_{q}
&+\norm{is\vvel+\tay\rotterm\vvel}_{q}
\\
&\qquad+\norm{\grad^2 \vvel}_{q}
+\rey\norm{\partial_1 \vvel}_{q}
+\norm{\grad \vpres}_{q}
\leq C \norm{g}_{q}
\end{aligned}
\label{est:Oseen.rot.res.R3}
\end{equation}
as well as
\begin{align}
\label{est:Oseen.rot.res.R3.grad}
\rey^{1/4}
\norm{\grad\vvel}_{4q/(4-q)}
&\leq C\norm{g}_{q}
\qquad\tif q<4,
\\
\label{est:Oseen.rot.res.R3.grad.Sobolev}
\norm{\grad\vvel}_{3q/(3-q)}
+\norm{\vpres}_{3q/(3-q)}
&\leq C\norm{g}_{q}
\qquad\tif q<3,
\\
\label{est:Oseen.rot.res.R3.fct}
\rey^{1/2}
\norm{\vvel}_{2q/(2-q)}
&\leq C\norm{g}_{q}
\qquad\tif q<2,
\\
\label{est:Oseen.rot.res.R3.fct.Sobolev}
\norm{\vvel}_{3q/(3-2q)}
&\leq C\norm{g}_{q}
\qquad\tif q<3/2
\end{align}
for a constant $C=C(q,\rey,\tay,s)>0$ given by
\begin{equation}
\label{eq:Oseen.rot.res.R3.const}
C=\begin{cases}
C_0 P\bp{\rey^2/\tay}
&\tif s\in\tay\Z,
\\
C_0 P\bp{\rey^2/\dist(s,\tay\Z)}
&\tif s\not\in\tay\Z,
\end{cases}
\end{equation}
where $C_0=C_0(q)>0$ is a constant only depending on $q$,
and $P(\theta)\coloneqq(1+\theta)^3$. 
Moreover, if $\np{\wvel,\wpres}\in\LRloc{1}(\R^3)^{3+1}$
is another distributional solution
to \eqref{sys:Oseen.rot.res.R3}, then the following holds:
\begin{enumerate}[label=\roman*.]
\item
If $\grad^2\wvel,\ is\wvel+\tay\rotterm\wvel,\,
\partial_1\wvel\in\LR{q}(\R^3)$,
then 
\[
\begin{aligned}
is\wvel+\tay\rotterm\wvel&=is\vvel+\tay\rotterm\vvel,
\\
\grad^2\wvel=\grad^2\vvel,
\qquad
&\partial_1\wvel=\partial_1\vvel,
\qquad
\grad\wpres=\grad\vpres.
\end{aligned}
\]
\item
If $q<2$ or $s\not\in\tay\Z$, and if 
$\wvel\in\LR{r}(\R^3)^3$ for some $r\in[1,\infty)$,
then $\wvel=\vvel$ and $\wpres=\vpres+c$ for some constant $c\in\R$.
\end{enumerate}
\end{thm}

For the analysis of \eqref{sys:Oseen.rot.res.R3}
we first study the auxiliary problem 
\begin{equation}\label{sys:Oseen.tp.mod.R3}
\begin{pdeq}
is\uvel
+\pt\uvel
- \Delta \uvel
-\rey\partial_1\uvel
+ \grad \upres 
&= f
&&\tin\torus\times\R^3,
\\
\Div\uvel
&=0 
&& \tin\torus\times\R^3.
\end{pdeq}
\end{equation}
As before, the torus group $\torus\coloneqq\R/\per\Z$
indicates time periodicity of all functions
with a given period $\per>0$.
Problem \eqref{sys:Oseen.tp.mod.R3}
has the advantage that,
in contrast to the original time-periodic problem \eqref{sys:Oseen.rot.tp},
the associated differential operator
has constant coefficients.
Therefore, one can readily derive a representation formula 
for its solution 
in terms of Fourier multipliers.
Since problem \eqref{sys:Oseen.tp.mod.R3}
is formulated in the locally compact abelian group $\grp\coloneqq\torus\times\R^3$,
we also use multiplier expressions in $\grp$.
One tool to derive $\LR{q}$ multiplier estimates
in this framework is
the so-called transference principle,
which goes back to de Leuuw \cite{Leeuw1965} and, 
in a generalized form, to Edwards and Gaudry \cite[Theorem B.2.1]{EdwardsGaudryBook}.
For an introduction how to employ this theory in the context of the Navier--Stokes equations,
we refer to \cite{EiterKyed_tplinNS_PiFbook}.
Here we use the following version of the transference principle.

\begin{thm}
\label{thm:TransferencePrinciple}
Let $\grp\coloneqq\torus\times\R^3$ and $\grph\coloneqq\R\times\R^3$.
For each $q\in(1,\infty)$ there exists a constant $C_q>0$
with the following property:
If a continuous function $M\colon\grph\to\C$ is 
an $\LR{q}\np{\grph}$ multiplier,
that is,
\[
\forall h\in\SR(\grph) : \quad
\norml{\iFT_{\grph}\bb{M\,\FT_{\grph}\nb{h}}}_{\LR{q}(\grph)}
\leq C_M \norm{h}_{\LR{q}(\grph)}
\]
for some $C_M>0$,
then the restriction $m\coloneqq\restriction{M}{\Z\times\R^3}$
is an $\LR{q}\np{\grp}$ multiplier such that
\[
\forall g\in\SR(\grp) : \quad
\norml{\iFT_{\grp}\bb{m\,\FT_{\grp}\nb{g}}}_{\LR{q}(\grp)}
\leq C_q C_M \norm{g}_{\LR{q}(\grp)}.
\]
\end{thm}

With this theorem one can reduce Fourier multipliers in the group $\grp=\torus\times\R^3$ 
to Fourier multipliers
in the Euclidean setting $\grph=\R\times\R^3$, 
where tools like the multiplier theorems by Mikhlin and Marcinkiewicz 
are available.
We use this method to 
derive $\LR{q}$ estimates of solutions to \eqref{sys:Oseen.tp.mod.R3}
in the context of the following existence theorem.

\begin{thm}\label{thm:Oseen.tp.mod.R3}
Let $q\in(1,\infty)$,
let $\rey,\,\per>0$ and $s\in\R$, and set $\tay\coloneqq\perf$.
For each $f\in\LR{q}(\torus\times\R^3)^3$ there exists 
a solution $\np{\uvel,\upres}$ with
\[
\uvel\in\WSR{1}{q}(\torus;\LRloc{q}(\R^3)^3)\cap\LR{q}(\torus;\WSRloc{2}{q}(\R^3)^3),
\qquad
\upres\in\LR{q}(\torus;\WSRloc{1}{q}(\R^3))
\]
to \eqref{sys:Oseen.tp.mod.R3}
that satisfies
\begin{equation}
\norm{\dist(s,\tay\Z)\,\uvel}_{q}
+\norm{is\uvel+\pt\uvel}_{q}
+\norm{\grad^2 \uvel}_{q}
+\rey\norm{\partial_1 \uvel}_{q}
+\norm{\grad \upres}_{q}
\leq C \norm{f}_{q}
\label{est:Oseen.tp.mod.R3}
\end{equation}
as well as
\begin{align}
\label{est:Oseen.tp.mod.R3.grad}
\rey^{1/4}
\norm{\grad\uvel}_{\LR{q}(\torus;\LR{4q/(4-q)}(\R^3))}
&\leq C\norm{f}_{q}
\qquad\tif q<4,
\\
\label{est:Oseen.tp.mod.R3.grad.Sobolev}
\norm{\grad\uvel}_{\LR{q}(\torus;\LR{3q/(3-q)}(\R^3))}
+\norm{\upres}_{\LR{q}(\torus;\LR{3q/(3-q)}(\R^3))}
&\leq C\norm{f}_{q}
\qquad\tif q<3,
\\
\label{est:Oseen.tp.mod.R3.fct}
\rey^{1/2}
\norm{\uvel}_{\LR{q}(\torus;\LR{2q/(2-q)}(\R^3))}
&\leq C\norm{f}_{q}
\qquad\tif q<2,
\\
\label{est:Oseen.tp.mod.R3.fct.Sobolev}
\norm{\uvel}_{\LR{q}(\torus;\LR{3q/(3-2q)}(\R^3))}
&\leq C\norm{f}_{q}
\qquad\tif q<3/2,
\end{align}
for the constant $C>0$ from \eqref{eq:Oseen.rot.res.R3.const}.
Moreover, if $\np{\wvel,\wpres}\in\LRloc{1}(\torus\times\R^3)^{3+1}$
is another distributional solution
to \eqref{sys:Oseen.tp.mod.R3}, then the following holds:
\begin{enumerate}[label=\roman*.]
\item
\label{it:Oseen.tp.mod.R3.uniqueness.i}
If $\grad^2\wvel,\ is\wvel+\pt\wvel,\,\partial_1\wvel\in\LR{q}(\torus\times\R^3)$,
then 
\[
is\wvel+\pt\wvel=is\uvel+\pt\uvel,
\qquad
\grad^2\wvel=\grad^2\uvel,
\qquad
\partial_1\wvel=\partial_1\uvel
\qquad
\grad\wpres=\grad\upres.
\]
\item
\label{it:Oseen.tp.mod.R3.uniqueness.ii}
If $q<2$ or $s\not\in\tay\Z$, and if 
$\wvel\in\LR{1}(\torus;\LR{r}(\R^3)^3)$ for some $r\in[1,\infty)$,
then $\uvel=\wvel$ and $\upres=\wpres+d$ for a (space-independent) function 
$d\colon\torus\to\R$.
\end{enumerate}
\end{thm}

\begin{proof}
We mainly follow the proof of 
\cite[Theorem 4.1]{Eiter2021_StokesResTPFlowRotating},
where a similar existence result was shown for \eqref{sys:Oseen.tp.mod.R3}
in the Stokes case $\rey=0$.
However, the derivation of the estimates 
\eqref{est:Oseen.tp.mod.R3}--\eqref{est:Oseen.tp.mod.R3.fct.Sobolev}
is more involved in the present case $\rey>0$.

At first, let $s\in\R$ and consider $\ell\in\Z$ such that
$\snorm{s-\tay\ell}\leq\tay/2$.
We set $\widetilde{s}=s-\tay\ell$ and $\widetilde{f}(t,x)=f(t,x)\e^{i\tay\ell t}$,
and assume that $\np{\tuvel,\tupres}$
is a solution to \eqref{sys:Oseen.tp.mod.R3} 
satisfying \eqref{est:Oseen.tp.mod.R3}--\eqref{est:Oseen.tp.mod.R3.fct.Sobolev}
with $s$ and $f$ replaced with $\widetilde{s}$ and $\widetilde{f}$,
respectively.
Then 
$\np{\uvel,\upres}$
with $\uvel(t,x)\coloneqq\tuvel(t,x)\e^{-i\tay\ell t}$
and $\upres(t,x)\coloneqq\tupres(t,x)\e^{-i\tay\ell t}$
satisfies the original problem \eqref{sys:Oseen.tp.mod.R3} 
and the corresponding estimates \eqref{est:Oseen.tp.mod.R3}--\eqref{est:Oseen.tp.mod.R3.fct.Sobolev}.
This shows that it suffices to only consider $s\in\R$ with $\snorm{s}\leq\tay/2$.

For $0<\snorm{s}\leq\tay/2$,
we first consider $f\in\SR(\grp)^3$,
where $\grp\coloneqq\torus\times\R^3$.
To derive a representation formula for $\upres$,
we compute the divergence of \eqrefsub{sys:Oseen.tp.mod.R3}{1},
which leads to $\Delta\upres=\Div f$
and, by means of the Fourier transform $\FT_\grp$,
the identity $-\snorm{\xi}^2\FT_\grp\nb{\upres}=i\xi\cdot\FT_\grp\nb{f}$.
This yields
\begin{equation}\label{eq:Oseen.tp.mod.R3.gradp}
\upres
=\iFT_{\grp}\Bb{\frac{-i\xi}{\snorm{\xi}^2}\FT_\grp\nb{f}},
\qquad
\grad\upres
=\iFT_{\grp}\Bb{\frac{\xi\otimes\xi}{\snorm{\xi}^2}\FT_\grp\nb{f}}.
\end{equation}
In particular, $\upres$ is well defined as a distribution in $\TDR(\grp)$ in this way, 
and the continuity of the Riesz transforms $\LR{q}(\R^3)\to\LR{q}(\R^3)$
implies
\begin{equation}\label{est:Oseen.tp.mod.R3.grad.Sobolevp}
\norm{\grad\upres}_{q}\leq C\norm{f}_{q}.
\end{equation}
Next we apply the Fourier transform to \eqrefsub{sys:Oseen.tp.mod.R3}{1}
and conclude
\begin{equation}\label{eq:Oseen.tp.mod.R3.u}
\uvel=\iFT_\grp\bb{m\,\FT_\grp\nb{f-\grad\upres}}
=\iFT_\grp\Bb{m\Bp{\idmatrix-\frac{\xi\otimes\xi}{\snorm{\xi}^2}}\FT_\grp\nb{f}}
\end{equation}
with
\[
m\colon\Z\times\R^3\to\R,\qquad
m(k,\xi)\coloneqq\frac{1}{is+i\tay k -i\rey\xi_1+\snorm{\xi}^2}.
\]
Since $0<\snorm{s}<\tay/2$,
the denominator of $m$ 
has no zeros $(k,\xi)\in\Z\times\R^3$ and is bounded,
so that $\uvel$ is a well-defined distribution in $\TDR(\grp)$.
Moreover, we derive the formulas 
\begin{align*}
is\uvel
&=\iFT_\grp\Bb{m_0\Bp{\idmatrix-\frac{\xi\otimes\xi}{\snorm{\xi}^2}}\FT_\grp\nb{f}},
&m_0(k,\xi)&\coloneqq\frac{is}{is+i\tay k-i\rey\xi_1 +\snorm{\xi}^2},
\\
\pt\uvel
&=\iFT_\grp\Bb{m_1\Bp{\idmatrix-\frac{\xi\otimes\xi}{\snorm{\xi}^2}}\FT_\grp\nb{f}},
&m_1(k,\xi)&\coloneqq\frac{i\tay k}{is+i\tay k-i\rey\xi_1 +\snorm{\xi}^2},
\\
\partial_j\partial_\ell\uvel
&=\iFT_\grp\Bb{m_{j\ell}\Bp{\idmatrix-\frac{\xi\otimes\xi}{\snorm{\xi}^2}}\FT_\grp\nb{f}},
&m_{j\ell}(k,\xi)&\coloneqq\frac{-\xi_j\xi_\ell}{is+i\tay k -i\rey\xi_1+\snorm{\xi}^2},
\\
\rey\partial_1\uvel
&=\iFT_\grp\Bb{m_2\Bp{\idmatrix-\frac{\xi\otimes\xi}{\snorm{\xi}^2}}\FT_\grp\nb{f}},
&m_2(k,\xi)&\coloneqq\frac{i\rey\xi}{is+i\tay k-i\rey\xi_1 +\snorm{\xi}^2}.
\end{align*}
For the $\LR{q}$ estimates, 
we now employ the transference principle from Theorem \ref{thm:TransferencePrinciple}.
We focus on $m_2$ in the following, the other multipliers can be treated in a similar fashion.
We introduce
$\cutoff\in\CRi(\R)$ with $0\leq\cutoff\leq1$ and
such that $\cutoff(x)=0$ for $\snorm{x}\leq 1/2$ and $\cutoff(x)=1$ for $\snorm{x}\geq1$,
and we define $M_2\colon\R\times\R^3\to\C$ by
\[
M_{2}(\eta,\xi)\coloneqq\frac{i\rey\xi_1\,\cutoff\bp{1+\frac{\tay\eta}{s}}}{N(\eta,\xi)},
\qquad
N(\eta,\xi)\coloneqq
is+i\tay \eta -i\rey\xi+\snorm{\xi}^2.
\]
The numerator vanishes in a neighborhood of the 
only zero $\np{\eta,\xi}=\np{-s/\tay,0}$
of the denominator $N$. 
Hence
$M_2$ is a well-defined continuous function with 
$\restriction{M_2}{\Z\times\R^3}=m_2$.
Moreover, $M_2$ vanishes for $\snorm{s+\tay\eta}\leq\snorm{s}/2$,
and for $\snorm{s+\tay\eta}\geq\snorm{s}/2$
we have
\[
\begin{aligned}
\snorm{s+\tay\eta}\geq 2\rey\snorm{\xi_1}
&\implies
\snorml{N(\eta,\xi)}
\geq
\snorm{s+\tay\eta -\rey\xi_1}\geq\half\snorm{s+\tay\eta},
\\
\snorm{s+\tay\eta}\leq 2\rey\snorm{\xi_1}
&\implies
\snorml{N(\eta,\xi)}
\geq\snorm{\xi}^2
\geq\frac{1}{4\rey^2}\snorm{s+\tay\eta}^2
\geq\frac{\snorm{s}}{8\rey^2}\snorm{s+\tay\eta},
\end{aligned}
\]
so that
\[
\frac{\snorm{s+\tay\eta}}{\snorm{N(\eta,\xi)}}
\leq C\bp{1+\frac{\rey^2}{\snorm{s}}}.
\]
Using this estimate, 
for $\snorm{s+\tay\eta}\geq\snorm{s}/2$
we further obtain
\[
\frac{\snorm{s}}{\snorm{N(\eta,\xi)}}
+\frac{\snorm{\tay\eta}}{\snorm{N(\eta,\xi)}}
+\frac{\snorm{\rey\xi_1}}{\snorm{N(\eta,\xi)}}
\leq C\bp{1+\frac{\rey^2}{\snorm{s}}}.
\]
With these estimates at hand, one shows
that
\[
\sup
\setcl{\snorml{
\eta^{\alpha}\xi^{\beta}\partial_\eta^{\alpha}\partial_\xi^\beta
M_2(\eta,\xi)
}}
{\alpha\in\set{0,1},\,
\beta\in\set{0,1}^3,\,
\np{\eta,\xi}\in\R\times\R^3
}
\leq C\Bp{1+\frac{\rey^2}{\snorm{s}}}^3.
\]
By the Marcinkiewicz multiplier theorem 
(see \cite[Corollary 5.2.5]{Grafakos1} for example)
we thus conclude
that $M_2$ is an $\LR{q}(\R\times\R^3)$ multiplier,
and the
transference principle (Theorem~\ref{thm:TransferencePrinciple})
implies that $m_2$ is
an $\LR{q}(\grp)$ multiplier with
\[
\norm{\iFT_{\grp}\bb{m_2\FT_{\grp}\nb{g}}}_{q}
\leq C\Bp{1+\frac{\rey^2}{\snorm{s}}}^3 \norm{g}_{q}
\]
for all $g\in\SR(\grp)$, where $C=C(q)$.
In the same way, we show 
that $m_0$, $m_1$ and $m_{j\ell}$ are multipliers with 
associated norms bounded by $C(1+\rey^2/\snorm{s})^3$.
By combining these estimates with the continuity of the Riesz transforms,
the above representation formulas
yield 
\[
\norm{is\uvel}_{q}+\norm{\pt\uvel}_{q}+\norm{\nabla^2\uvel}_{q}+\rey\norm{\partial_1\uvel}_{q}
\leq C_0\Bp{1+\frac{\rey^2}{\snorm{s}}}^3\norm{f}_{q}
\]
with a constant $C_0=C_0(q)$.
Since $0<\snorm{s}<\tay/2$, the combination of this estimate with \eqref{est:Oseen.tp.mod.R3.grad.Sobolevp}
yields \eqref{est:Oseen.tp.mod.R3}.
Finally, \eqref{est:Oseen.tp.mod.R3.grad.Sobolev} and \eqref{est:Oseen.tp.mod.R3.fct.Sobolev}
follow from Sobolev's inequality,
and \eqref{est:Oseen.tp.mod.R3.grad} and \eqref{est:Oseen.tp.mod.R3.fct}
are implied by Lemma \ref{lem:embedding.oseen}.
In summary, 
for $f\in\SR(\grp)$ we have now constructed a solution to \eqref{sys:Oseen.tp.mod.R3}
with the desired properties.
A classical approximation argument
based on the estimates \eqref{est:Oseen.tp.mod.R3}--\eqref{est:Oseen.tp.mod.R3.fct.Sobolev} 
finally yields the existence of a solution for any $f\in\LR{q}(\grp)$. 

In the case $s=0$,
problem \eqref{sys:Oseen.tp.mod.R3}
reduces to the classical time-periodic Oseen system.
Solutions $\np{\uvel,\upres}$ to this problem
were established in \cite{Kyed_mrtpns},
and the validity of the \textit{a priori} estimate 
\eqref{est:Oseen.tp.mod.R3}
with a constant $C=C_0 P(\rey^2/\tay)$
was derived in the proof of 
\cite[Theorem 5.1]{EiterKyed_ViscousFlowAroundRigidBodyPerformingTPMotion_2021},
where $P$ is a polynomial.
Arguing as above, one can show that $P$ can be chosen in the claimed form.

For the uniqueness assertion,
we let
$\np{\tuvel,\tupres}\coloneqq\np{\uvel-\wvel,\upres-\wpres}\in\LRloc{1}(\grp)^{3+1}$,
which is a solution  
to \eqref{sys:Oseen.tp.mod.R3} with $f=0$.
Computing the divergence of both sides of \eqrefsub{sys:Oseen.tp.mod.R3}{1},
we conclude $\Delta\tupres=0$ and, in particular,
$\supp\FT_\grp\nb{\tupres}\subset\Z\times\set{0}$.
An application of $\FT_\grp$ to \eqrefsub{sys:Oseen.tp.mod.R3}{1} thus
leads to
$(is+i\tay k + \snorm{\xi}^2-i\rey\xi_1)\FT_\grp\nb{\tuvel}=-i\xi\FT_\grp\nb{\tupres}$,
and we deduce
\[
\supp\bb{(is+i\tay k + \snorm{\xi}^2-i\rey\xi_1)\FT_\grp\nb{\tuvel}}
\subset\Z\times\set{0}.
\]
Since $is+i\tay k + \snorm{\xi}^2-i\rey\xi_1$ can only vanish for $\xi=0$, 
we conclude
$\supp\FT_\grp\nb{\tuvel}
\subset\Z\times\set{0}$,
so that
$\supp\FT_{\R^3}\nb{\tuvel}(t,\cdot)\subset\set{0}$
for a.a.~$t\in\torus$ due to $\tuvel\in\LRloc{1}(\grp)$.
Hence, $\tuvel(t,\cdot)$ is a polynomial 
for a.a.~$t\in\torus$. 
In the same way we show that $\tupres(t,\cdot)$ is a polynomial for a.a.~$t\in\torus$.
In case \ref{it:Oseen.tp.mod.R3.uniqueness.i}~
we additionally have $\grad^2\tuvel,\,\partial_1\tuvel,\,\grad\tupres\in\LR{q}(\grp)$,
which is only possible if $\grad^2\tuvel=0$ and $\partial_1\tuvel=\grad\tupres=0$.
In virtue of \eqrefsub{sys:Oseen.tp.mod.R3}{1}, 
this also implies $is\tuvel+\pt\tuvel=0$.
This shows the statement in case \ref{it:Oseen.tp.mod.R3.uniqueness.i}
In case \ref{it:Oseen.tp.mod.R3.uniqueness.ii}~
we have that $\tuvel$ is a polynomial with $\tuvel\in\LR{1}(\torus;\LR{r_0}(\R^3)^3+\LR{r}(\R^3)^3)$
where $r_0=2q/(2-q)$ if $q<2$, and $r_0=q$ if $s\not\in\tay\Z$.
This is only possible if $\tuvel=0$,
and returning to \eqrefsub{sys:Oseen.tp.mod.R3}{1}, we also conclude $\grad\tupres=0$.
In total, this completes the proof.
\end{proof}

Next we introduce the rotation term and
study the problem
\begin{equation}\label{sys:Oseen.tp.rot.mod.R3}
\begin{pdeq}
is\uvel
+\pt\uvel
+\tay\rotterm\uvel
- \Delta \uvel
-\rey\partial_1\uvel
+ \grad \upres 
&= f
&&\tin\torus\times\R^3,
\\
\Div\uvel
&=0 
&& \tin\torus\times\R^3
\end{pdeq}
\end{equation}
in the case that $\tay=\perf$.
By means of a suitable coordinate transform, 
we derive well-posedness of \eqref{sys:Oseen.tp.rot.mod.R3}
by a reduction to problem \eqref{sys:Oseen.tp.mod.R3}.

\begin{thm}\label{thm:Oseen.tp.rot.mod.R3}
Let $q\in(1,\infty)$ and $\rey,\,\per>0$, and let $\tay=\perf$.
For each $f\in\LR{q}(\torus\times\R^3)^3$ there exists 
a solution $\np{\uvel,\upres}$ with
\[
\uvel\in\WSR{1}{q}(\torus;\LRloc{q}(\R^3)^3)\cap\LR{q}(\torus;\WSRloc{2}{q}(\R^3)^3),
\qquad
\upres\in\LR{q}(\torus;\WSRloc{1}{q}(\R^3))
\]
to \eqref{sys:Oseen.tp.rot.mod.R3}
that satisfies
\begin{equation}
\begin{aligned}
\norm{\dist(s,\tay\Z)\,\uvel}_{q}
&+\norm{is\uvel+\pt\uvel+\tay\rotterm\uvel}_{q}
\\
&\qquad\qquad\quad+\norm{\grad^2 \uvel}_{q}
+\rey\norm{\partial_1 \uvel}_{q}
+\norm{\grad \upres}_{q}
\leq C \norm{f}_{q}
\end{aligned}
\label{est:Oseen.tp.rot.mod.R3}
\end{equation}
as well as
\eqref{est:Oseen.tp.mod.R3.grad}--\eqref{est:Oseen.tp.mod.R3.fct.Sobolev}
for the constant $C>0$ from \eqref{eq:Oseen.rot.res.R3.const}.
Moreover, if $\np{\wvel,\wpres}\in\LRloc{1}(\torus\times\R^3)^{3+1}$
is another distributional solution
to \eqref{sys:Oseen.tp.rot.mod.R3}, then the following holds:
\begin{enumerate}[label=\roman*.]
\item
If $\grad^2\wvel,\ is\wvel+\pt\wvel+\tay\rotterm\wvel,\,
\partial_1\wvel\in\LR{q}(\torus\times\R^3)$,
then 
\[
\begin{aligned}
is\wvel+\pt\wvel+\tay\rotterm\wvel&=is\uvel+\pt\uvel+\tay\rotterm\uvel,
\\
\grad^2\wvel=\grad^2\uvel,
\qquad
&\partial_1\wvel=\partial_1\uvel
\qquad
\grad\wpres=\grad\upres.
\end{aligned}
\]
\item
If $q<2$ or $s\not\in\tay\Z$, and if 
$\wvel\in\LR{1}(\torus;\LR{r}(\R^3)^3)$ for some $r\in[1,\infty)$,
then $\uvel=\wvel$ and $\upres=\wpres+d$ for a (space-independent) function 
$d\colon\torus\to\R$.
\end{enumerate}
\end{thm}

\begin{proof}
The proof is based on the idea 
to absorb the rotational term $\tay\rotterm{\uvel}$ into the time derivative
by the coordinate transform
arising from the rotation matrix 
\begin{equation}\label{eq:Q.def}
\rotmatrix_\tay(t)\coloneqq\begin{pmatrix}
1 & 0 & 0\\
0 & \cos(\tay t) & -\sin(\tay t)\\
0 & \sin(\tay t) & \cos(\tay t)
\end{pmatrix}.
\end{equation}
Let $f\in\LR{q}(\torus\times\R^3)^3$ and define the vector field $\widetilde f$ by
\[
\widetilde{f}(t,x)
\coloneqq\rotmatrix_\tay(t) f(t,\rotmatrix_\tay(t)^\transpose x).
\]
Then $\widetilde{f}\in\LR{q}(\torus\times\R^3)^3$ 
since $\torus=\R/\per\Z$ with $\per=\frac{2\pi}{\tay}$.
By Theorem \ref{thm:Oseen.tp.mod.R3}
there exists a solution 
$\np{\tuvel,\tupres}$
to \eqref{sys:Oseen.tp.mod.R3} (with $f$ replaced by $\widetilde f$),
which satisfies the estimates \eqref{est:Oseen.tp.mod.R3}--\eqref{est:Oseen.tp.mod.R3.fct.Sobolev}.
We now define the $\per$-time-periodic functions
\[
\uvel(t,x)\coloneqq\rotmatrix_\tay(t)^\transpose\tuvel(t,\rotmatrix_\tay(t) x),
\qquad
\upres(t,x)\coloneqq\tupres(t,\rotmatrix_\tay(t) x).
\]
Since 
$\ddt\nb{{\rotmatrix}_\tay(t)x}
= \tay\eone\wedge \nb{\rotmatrix_\tay(t)x}
= \rotmatrix_\tay(t)\nb{\tay\eone\wedge x}$ 
for any $x\in\R^3$, a direct computation shows
that $\np{\uvel,\upres}$ is a solution to \eqref{sys:Oseen.tp.rot.mod.R3}
and obeys the estimates \eqref{est:Oseen.tp.rot.mod.R3} and 
\eqref{est:Oseen.tp.mod.R3.grad}--\eqref{est:Oseen.tp.mod.R3.fct.Sobolev}.

For uniqueness, 
we use the above transformation to obtain solutions
$\np{\tuvel,\tupres}$ and $\np{\twvel,\twpres}$ 
to \eqref{sys:Oseen.tp.mod.R3} with the same right-hand side $\widetilde f$.
The statement is then a consequence of
 follows from Theorem \ref{thm:Oseen.tp.mod.R3}.
\end{proof}

Observe that, by simply considering 
$s=0$ in \eqref{sys:Oseen.tp.rot.mod.R3},
we would obtain the original time-periodic problem 
\eqref{sys:Oseen.rot.tp},
and Theorem \ref{thm:Oseen.tp.rot.mod.R3}
yields existence of a unique solution.
However, in Theorem \ref{thm:Oseen.tp.rot.mod.R3}
we required $\tay=\perf$,
and the presented proof does not allow
to choose $\tay$ and $\per$ independently.
To make this possible,
we first consider time-independent solutions 
to \eqref{sys:Oseen.tp.rot.mod.R3},
which are solutions to the resolvent problem
\eqref{sys:Oseen.rot.res.R3}.
In this way, we conclude the proof of Theorem \ref{thm:Oseen.rot.res.R3}.

\begin{proof}[Proof of Theorem \ref{thm:Oseen.rot.res.R3}]
Set $\torus\coloneqq\R/\per\Z$ with $\per=\frac{2\pi}{\tay}$,
let $g\in\LR{q}(\R^3)^3$
and define $f(t,x)\coloneqq g(x)$. 
Then $f\in\LR{q}(\torus\times\R^3)^3$,
and there exists a solution $\np{\uvel,\upres}$
to \eqref{sys:Oseen.tp.rot.mod.R3} 
by Theorem \ref{thm:Oseen.tp.rot.mod.R3}.
Computing the time means
\[
\vvel(x)\coloneqq\int_\torus\uvel(t,x)\,\dt,
\qquad
\vpres(x)\coloneqq\int_\torus\upres(t,x)\,\dt,
\]
we obtain a solution $\np{\vvel,\vpres}$
to \eqref{sys:Oseen.rot.res.R3},
and estimates 
\eqref{est:Oseen.rot.res.R3}--\eqref{est:Oseen.rot.res.R3.fct.Sobolev}
follow directly from 
\eqref{est:Oseen.tp.rot.mod.R3}, 
\eqref{est:Oseen.tp.mod.R3.grad}--\eqref{est:Oseen.tp.mod.R3.fct.Sobolev}.
Since every solution to \eqref{sys:Oseen.rot.res.R3}
is a (time-independent) 
solution to \eqref{sys:Oseen.tp.rot.mod.R3},
the uniqueness statement follows directly from 
Theorem \ref{thm:Oseen.tp.rot.mod.R3}.
\end{proof}

\section{The resolvent problem in an exterior domain}
\label{sec:resprob.extdom}

After having established well-posedness 
of the resolvent problem \eqref{sys:Oseen.rot.res.R3}
in $\R^3$,
we next consider the corresponding problem
in an exterior domain $\Omega\subset\R^3$,
that is,
the resolvent problem \eqref{sys:Oseen.rot.res}
with a purely imaginary resolvent parameter $is$, $s\in\R$.
We first address the question of uniqueness 
of solutions.

\begin{lem}\label{lem:Oseen.rot.res.uniqueness}
Let $\Omega\subset\R^3$ be an exterior domain of class $\CR{1,1}$.
Let $\rey\geq 0$, $\tay>0$, $s\in\R$, 
and let $\np{\vvel,\vpres}$ be a distributional solution to \eqref{sys:Oseen.rot.res}
with $g=0$ and
$\grad^2\vvel,\,\partial_1\vvel,\,\grad\vpres \in\LR{q}(\Omega)$
for some $q\in(1,\infty)$ and $\vvel\in\LR{r}(\Omega)$ for some $r\in(1,\infty)$.
Then 
$\vvel=0$ and $\vpres$ is constant.
\end{lem}

\begin{proof}
For $\rey=0$, the statement was shown in \cite[Lemma 5.1]{Eiter2021_StokesResTPFlowRotating}.
For $\rey>0$ one can follow the proof of
\cite[Lemma 5.6]{EiterKyed_ViscousFlowAroundRigidBodyPerformingTPMotion_2021}, 
which treats the case $s\in\tay\Z$.
Therefore, we only sketch the main arguments here.
One first employs a cut-off argument
that leads to a Stokes problem in a bounded domain
and to the resolvent problem \eqref{sys:Oseen.rot.res.R3}
in the whole space,
both with error terms on the right-hand side.
Using elliptic regularity 
of the Stokes problem
and regularity properties 
for \eqref{sys:Oseen.rot.res.R3}
established in Theorem \ref{thm:Oseen.rot.res.R3},
one can show that
\[
\begin{aligned}
&\forall r\in(1,2): 
&&
is\vvel+\rottermsimple{\vvel},\,\grad^2\vvel,\,\grad\vpres
\in\LR{r}(\Omega),
\\
&\forall r\in\bp{\frac{3}{2},6}: 
&&
\grad\vvel\in\LR{r}(\Omega), 
\\
&\forall r\in(3,\infty) : 
&&
\vvel\in\LR{r}(\Omega).
\end{aligned}
\]
These regularities allow us to multiply 
\eqrefsub{sys:Oseen.rot.res}{1}
the complex conjugate of $\vvel$,
to integrate the resulting identity over $\Omega_R$,
and to pass to the limit $R\to\infty$.
This leads to 
\[
0 
= is\int_{\Omega}\snorm{\vvel}^2\,\dx
+\int_{\Omega}\snorm{\grad\vvel}^2 \,\dx,
\]
so that $\grad\vvel=0$,
which implies $\vvel=0$ in view of \eqrefsub{sys:Oseen.rot.res}{3}.
From \eqrefsub{sys:Oseen.rot.res}{1}
we finally conclude $\grad\vpres=0$.
\end{proof}

Suitable \textit{a priori} estimates 
for solutions to \eqref{sys:Oseen.rot.res}
can be derived by a similar cut-off procedure,
which first leads to the following intermediate result.
For simplicity, we only consider the case $q<2$,
where estimates of $\vvel$ and $\grad\vvel$ are available.

\begin{lem}\label{lem:Oseen.rot.res.estimates.pert}
Let $\Omega\subset\R^3$ be an exterior domain of class $\CR{1,1}$,
and $\rey,\tay>0$ and $s\in\R$. 
Let $q\in(1,2)$ and $g\in\LR{q}(\Omega)^3$.
Consider a solution $\np{\vvel,\vpres}$ 
to \eqref{sys:Oseen.rot.res} that satisfies
\[
is\vvel+\tay\rotterm\vvel, \
\grad^2\vvel,\
\partial_1\vvel,\
\grad\vpres
\in\LR{q}(\Omega)^3
\]
and $\vvel\in\LR{r}(\Omega)^3$, $\vpres\in\LR{\overline{r}}(\Omega)$ for some 
$r,\overline{r}\in(1,\infty)$.
Fix $R>0$ such that $\partial\Omega\subset\ball_{R}$.
Then $\np{\vvel,\vpres}$ satisfies the estimate
\begin{equation}
\begin{aligned}
&\norm{\dist(s,\tay\Z) \,\vvel}_q
+\norm{is\vvel+\tay\rotterm\vvel}_{q}
+\norm{\grad^2 \vvel}_{q}
+\rey\norm{\partial_1 \vvel}_{q}
+\norm{\grad \vpres}_{q}
\\
&\qquad\qquad\qquad
+\rey^{1/4}\norm{\grad\vvel}_{4q/(4-q)}
+\rey^{1/2}\norm{\vvel}_{2q/(2-q)}
+\norm{\vpres}_{3q/(3-q)}
\\
&\qquad\qquad\qquad\qquad
\leq C \bp{
\norm{g}_{q}+ \np{1+\rey+\tay}\norm{\vvel}_{1,q;\Omega_R}+\norm{\vpres}_{q;\Omega_R}
+\snorm{s}\norm{\vvel}_{-1,q;\Omega_R}
}
\end{aligned}
\label{est:Oseen.rot.res.extdom.pert}
\end{equation}
for a constant $C>0$
given by \eqref{eq:Oseen.rot.res.R3.const}
with $C_0=C_0(q,\Omega,R)>0$ and  $P(\theta)\coloneqq(1+\theta)^3$.
\end{lem}

\begin{proof}
The estimate
can be shown by a classical cut-off procedure.
We skip the details here
and refer to 
\cite[Lemma 5.7]{EiterKyed_ViscousFlowAroundRigidBodyPerformingTPMotion_2021},
where the special case $s\in\tay\Z$ was considered.
In the present situation one may proceed in the very same way
by invoking estimates \eqref{est:Oseen.rot.res.R3},
\eqref{est:Oseen.rot.res.R3.grad} and \eqref{est:Oseen.rot.res.R3.fct}
as well as the uniqueness result from Lemma \ref{lem:Oseen.rot.res.uniqueness}.
\end{proof}

Next we show how to omit the error terms on the right-hand side of
\eqref{est:Oseen.rot.res.extdom.pert}
by means of a compactness argument
in order to derive estimate \eqref{est:Oseen.rot.res}.
Again we want to keep track of the constant
in this \textit{a priori} estimate
and its dependence on the various parameters.
Observe that the dependence on $\rey$ can only be avoided 
if $q<3/2$.

\begin{lem}\label{lem:Oseen.rot.res.estimates}
In the situation of Lemma \ref{lem:Oseen.rot.res.estimates.pert},
the solution $\np{\vvel,\vpres}$ satisfies the estimate
\eqref{est:Oseen.rot.res}
for a constant $C=C(q,\Omega,\tay,s,\rey)>0$.
If $\tay\in(0,\tay_0]$ for some $\tay_0>0$,
and if
\[
\rey^2\leq\theta\min\setcl{\snorm{s-\tay k}}{k\in\Z,\,s\neq\tay k},
\]
for some $\theta>0$,
then we can choose $C$ such that $C=C(q,\Omega,\tay_0,\theta,\rey)>0$.
If additionally $q\in(1,3/2)$, then $C$ can be chosen such that $C=C(q,\Omega,\tay_0,\theta)>0$.
In particular, then $C$ is independent of $\tay$, $\rey$ and $s$.
\end{lem}

\begin{proof}
We perform a contradiction argument.
Consider the case $q\in(1,3/2)$ at first,
and assume that there is no constant $C$ with the claimed properties.
Then there exist sequences 
$\np{s_j}\subset\R$, $\np{\tay_j}\subset(0,\tay_0]$,
$\np{\rey_j}\subset(0,\infty)$
with 
\begin{equation}\label{est:reyj}
\rey_j^2\leq
\theta\min\setcl{\snorm{s_j-\tay_j k}}{k\in\Z,\,s_j\neq\tay_j k}
\end{equation}
and sequences
$\np{\vvel_j}\subset\WSRloc{2}{q}(\Omega)^3$, 
$\np{\vpres_j}\subset\WSRloc{1}{q}(\Omega)$,
$\np{g_j}\subset\LR{q}(\Omega)^3$
such that
\begin{equation}\label{sys:Oseen.rot.res.seq}
\begin{pdeq}
is_j\vvel_j+\tay_j\rotterm{\vvel_j} 
- \Delta \vvel_j
-\rey_j\partial_1\vvel_j
+ \grad \vpres_j
 &= g_j
&& \tin \Omega, \\
\Div\vvel_j&=0
&& \tin\Omega, \\
\vvel_j&=0
&& \ton \partial\Omega
\end{pdeq}
\end{equation}
as well as
\begin{equation}\label{eq:Oseen.rot.res.extdom.seqnorm}
\begin{aligned}
&\norm{\dist(s_j,\tay_j\Z) \,\vvel_j}_q
+\norm{is_j\vvel_j{+}\tay_j\rotterm{\vvel_j}}_{q}
+\norm{\grad^2 \vvel_j}_{q}
+\rey_j\norm{\partial_1\vvel_j}_{q}
\\
&\qquad\qquad\quad
+\norm{\grad \vpres_j}_{q}
+\rey^{1/4}\norm{\grad\vvel_j}_{4q/(4-q)}
+\rey^{1/2}\norm{\vvel_j}_{2q/(2-q)}
+\norm{\vpres_j}_{3q/(3-q)}
=1
\end{aligned}
\end{equation}
and 
\[
\lim_{j\to\infty}\norm{g_j}_{q}
=0.
\]
Moreover, there are sequences $(r_j),\,(\overline{r}_j)\subset(1,\infty)$
such that
$\vvel_j\in\LR{r_j}(\Omega)^3$, $\vpres_j\in\LR{\overline{r}_j}(\Omega)$.
Here we used that the left-hand side of \eqref{eq:Oseen.rot.res.extdom.seqnorm} 
is finite due to Lemma \ref{lem:Oseen.rot.res.estimates.pert},
so that it can be normalized to $1$.
Note that \eqref{est:reyj} implies $\rey_j^2\leq\theta\tay_j^2\leq\theta\tay_0^2$ for all $j\in\N$.
Upon the choice of a subsequence, we thus conclude
the convergence of the sequences
$\tay_j\to\tay\in[0,\tay_0]$,
$\rey_j\to\rey\in[0,\sqrt{\theta\tay}]$, 
$s_j\to s\in[-\infty,\infty]$ 
and $\dist(s_j,\tay_j\Z)\to\delta\in[-\tay/2,\tay/2]$
as $j\to\infty$.
For the moment, fix $R>0$ with $\partial\Omega\subset\ball_R$.
In virtue of \eqref{eq:Oseen.rot.res.extdom.seqnorm}
and the estimate
\begin{equation}\label{est:isjvj}
\begin{aligned}
&\norm{is_j\vvel_j}_{q;\Omega_R}\\
&\quad\leq \norm{is_j\vvel_j+\tay_j\rotterm{\vvel_j}}_{q;\Omega_R}
+\norm{\tay_j\rotterm{\vvel_j}}_{q;\Omega_R}
\\
&\quad\leq \norm{is_j\vvel_j+\tay_j\rotterm{\vvel_j}}_{q} 
+ \tay_0\bp{\norm{\vvel_j}_{q;\Omega_R}
+ R\norm{\grad\vvel_j}_{q;\Omega_R}},
\end{aligned}
\end{equation}
the sequences 
$(\restriction{is\vvel_j}{\Omega_R})$, 
$(\restriction{\vvel_j}{\Omega_R})$ and
$(\restriction{\vpres_j}{\Omega_R})$
are bounded in 
$\LR{q}(\Omega_R)$, $\WSR{2}{q}(\Omega_R)$ and $\WSR{1}{q}(\Omega_R)$, respectively.
Upon selecting suitable subsequences,
we thus infer the existence of
$\wvel\in\LRloc{q}(\Omega)^3$, 
$\vvel\in\WSRloc{2}{q}(\Omega)^3$ 
and $\vpres\in\WSRloc{1}{q}(\Omega)$
such that
\[
is_j\vvel_j\wto\wvel 
\quad\tin\LR{q}(\Omega_R),
\qquad
\vvel_j\wto\vvel
\quad\tin\WSR{2}{q}(\Omega_R),
\qquad
\vpres_j\wto\vpres
\quad\tin\WSR{1}{q}(\Omega_R).
\]
By a Cantor diagonalization argument one can find the limit functions $\wvel$, $\vvel$, $\vpres$
that are independent of the radius $R>0$.
Moreover, due to the uniform bounds in \eqref{eq:Oseen.rot.res.extdom.seqnorm},
we can assume weak convergence in the norms on the left-hand side 
of \eqref{eq:Oseen.rot.res.extdom.seqnorm},
which yields
\[
\begin{aligned}
&\norm{\delta\vvel}_q
+\norm{\wvel+\tay\rotterm{\vvel}}_{q}
+\norm{\grad^2 \vvel}_{q}
+\rey\norm{\partial_1\vvel}_{q}
+\norm{\grad \vpres}_{q}
\\
&\qquad\qquad\qquad\qquad
+\rey^{1/4}\norm{\grad\vvel}_{4q/(4-q)}
+\rey^{1/2}\norm{\vvel}_{2q/(2-q)}
+\norm{\vpres}_{3q/(3-q)}
\leq 1,
\end{aligned}
\]
and Sobolev's inequality yields $\vvel\in\LR{3q/(3-2q)}(\Omega)$
since $q<3/2$.
Moreover, we can pass to the limit $j\to\infty$
in \eqref{sys:Oseen.rot.res.seq}
and conclude
\begin{equation}\label{sys:Oseen.rot.res.extdom.limit}
\begin{pdeq}
\wvel +\tay\rotterm\vvel
- \Delta \vvel
-\rey\partial_1\vvel
+ \grad \vpres 
&= 0
&&\tin\Omega,
\\
\Div\vvel
&=0 
&& \tin\Omega,
\\
\vvel
&=0 
&& \tin\partial\Omega.
\end{pdeq}
\end{equation}
We now distinguish the following cases:
\begin{enumerate}[label=\roman*.]
\item
If $\snorm{s}<\infty$ and $\tay=0$, then $\wvel=is\vvel$ and $\rey=0$
and \eqref{sys:Oseen.rot.res.extdom.limit}
simplifies to the classical Stokes resolvent problem with resolvent parameter $is$.
If $s\neq0$, this yields $is\vvel=\Delta\vvel-\grad\vpres\in\LR{q}(\Omega)$, 
so that $\vvel\in\WSR{2}{q}(\Omega)$.
Uniqueness in this functional framework is well known,
so that $\vvel=\grad\vpres=0$; see \cite{FarwigSohr1994} for example.
If $s=0$, then \eqref{sys:Oseen.rot.res.extdom.limit} 
coincides with the steady-state Stokes problem, and $\vvel\in\LR{3q/(3-2q)}(\Omega)$
implies $\vvel=\grad\vpres=0$;
see \cite[Theorem V.4.6]{GaldiBookNew} 
for example.
\item
If $\snorm{s}<\infty$ and $\tay>0$,
then $\wvel=is\vvel$ and \eqref{sys:Oseen.rot.res.extdom.limit}
coincides with \eqref{sys:Oseen.rot.res} with $g=0$ and $\rey\geq0$.
From Lemma \ref{lem:Oseen.rot.res.uniqueness} and $\vvel\in\LR{3q/(3-2q)}(\Omega)$
we now conclude $\vvel=\grad\vpres=0$.
\item
If $\snorm{s}=\infty$,
we note that estimates \eqref{eq:Oseen.rot.res.extdom.seqnorm} and \eqref{est:isjvj}
imply $\norm{\vvel_j}_{q;\Omega_R}\leq C /\snorm{s_j}$
for some $R$-dependent constant $C$.
Passing to the limit $j\to\infty$ 
and employing that $R$ was arbitrary,
we deduce $\vvel=0$ and \eqref{sys:Oseen.rot.res.extdom.limit}
reduces to $\wvel+\grad\vpres=0$,
which, in particular, yields $\wvel\in\LR{q}(\Omega)$.
Since we also have $\Div\wvel=0$ and $\restriction{\wvel}{\partial\Omega}=0$,
this equality 
corresponds to the Helmholtz decomposition in $\LR{q}(\Omega)$ of the zero function,
which is unique, so that $\wvel=\grad\vpres=0$.
\end{enumerate}
All three cases lead to
$\wvel=\vvel=\grad\vpres=0$,
and $\vpres\in\LR{3q/(3-q)}(\Omega)$ further implies $\vpres=0$.
By Lemma \ref{lem:Oseen.rot.res.estimates.pert}
we further have \eqref{est:Oseen.rot.res.extdom.pert}
with $\vvel$, $\vpres$, $g$ replaced with $\vvel_j$, $\vpres_j$, $g_j$
and the constant $C$ given by $C=C_0(1+\theta)^3$.
In particular, $C$ is uniform in $j\in\N$.
Employing \eqref{eq:Oseen.rot.res.extdom.seqnorm}
and passing to the limit $j\to\infty$ in this inequality,
we thus deduce
\[
1
\leq C \bp{
\np{1+\rey+\tay}\norm{\vvel}_{1,q;\Omega_{R}}+\norm{\vpres}_{q;\Omega_{R}}
+\norm{\wvel}_{-1,q;\Omega_{R}}
}
=0,
\]
which is a contradiction.
This finishes the proof for $q\in(1,3/2)$.
 
In the case $q\in(1,2)$, where the constant $C$ may depend on $\rey$,
the proof follows nearly the same lines,
but we consider a fixed value $\rey>0$ instead of a sequence $(\rey_j)$.
Therefore, the case $\rey=0$ cannot occur in the limit system \eqref{sys:Oseen.rot.res.extdom.limit}
and we have $\vvel\in\LR{2q/(2-q)}(\Omega)$,
which is sufficient to deduce $\wvel=\vvel=\grad\vpres=0$ as before
and to conclude the contradiction argument.
\end{proof}

We next show the existence of a solution to the resolvent problem \eqref{sys:Oseen.rot.res}
for a smooth right-hand side $g$.

\begin{lem}
\label{lem:Oseen.rot.res.exist.smooth}
Let $\Omega\subset\R^3$ be an exterior domain of class $\CR{3}$.
Let $\rey,\,\tay>0$, $s\in\R$ and $g\in\CRci(\Omega)^3$.
Then there exists a solution $\np{\vvel,\vpres}$ to \eqref{sys:Oseen.rot.res}
with
\[
\forall q\in(1,2):\ 
\np{\vvel,\vpres}\in\Xreys(\Omega)\times\Ys(\Omega).
\]
\end{lem}

\begin{proof}
In the case $s\in\tay\Z$, existence was shown in
\cite[Lemma 5.11]{EiterKyed_ViscousFlowAroundRigidBodyPerformingTPMotion_2021}
in full detail based on energy estimates and an ``invading domains'' technique
as well as the $\LR{q}$ estimates \eqref{est:Oseen.rot.res},
which we established in 
Lemma \ref{lem:Oseen.rot.res.estimates}.
The proof for general $s\in\R$
follows along the same lines,
and we only give a rough sketch here.

First of all, we choose $R>0$ such that $\partial\Omega\subset\ball_R$.
For $m\in\N$ with $m>R$ we consider the resolvent problem \eqref{sys:Oseen.rot.res} in 
the bounded domain $\Omega_m=\Omega\cap\ball_m$,
which is given by
\[
\begin{pdeq}
is\vvel_m+\tay\rotterm{\vvel_m} 
- \Delta \vvel_m
-\rey\partial_1\vvel_m
+ \grad \vpres_m
 &= g
&& \tin\Omega_m, \\
\Div\vvel_m&=0
&& \tin\Omega_m, \\
\vvel_m&=0
&& \ton \partial\Omega_m.
\end{pdeq}
\]
By formally testing with the complex conjugates of $\vvel_m$ and $\projhm\Delta\vvel_m$,
where $\projhm$ denotes the Helmholtz projection in $\LR{2}(\Omega_m)$,
one can then derive the \textit{a priori} estimates
\[
\begin{aligned}
\norm{\vvel_m}_{6;\Omega_m}+\norm{\grad\vvel_m}_{2;\Omega_m}
&\leq C \norm{g}_{6/5},
\\
\norm{\projhm\Delta\vvel_m}_{2;\Omega_m}
&\leq C \bp{\norm{g}_{6/5}+\norm{g}_{2}},
\end{aligned}
\]
where the constant $C>0$ is independent of $m$.
In order to derive a uniform estimate on the full second-order norm,
we employ the inequality
\[
\norm{\grad^2\wvel}_{2;\Omega_m}
\leq C \np{\norm{\projhm\Delta\wvel}_{2;\Omega_m}
+\norm{\grad\wvel}_{2;\Omega_m}}
\]
for all $\wvel\in\WSRN{1}{2}(\Omega_m)^3\cap\WSR{2}{2}(\Omega_m)^3$ with $\Div\wvel=0$.
Since we assumed $\partial\Omega\in\CR{3}$,
the constant $C$ can be chosen independently of $m$;
see \cite[Lemma 1]{Heywood1980}.
Using these formal \textit{a priori} estimates,
one can then apply a Galerkin method, 
based on a basis of eigenfunctions of the Stokes operator on the bounded domain $\Omega_m$,
to conclude the existence of a solution $\np{\vvel_m,\vpres_m}$,
which satisfies the \textit{a priori} estimate
\[
\norm{\vvel_m}_{6;\Omega_m}+\norm{\grad\vvel_m}_{1,2;\Omega_m}
\leq C \np{\norm{g}_{6/5}+\norm{g}_{2}},
\]
where $C$ is independent of $m$.
After multiplication with suitable cut-off functions,
one can then pass to the limit $m\to\infty$,
which leads to 
a solution $\np{\vvel,\vpres}$ to the original resolvent problem 
\eqref{sys:Oseen.rot.res}.
Finally, another cut-off argument that uses
the uniqueness properties from Lemma \ref{lem:Oseen.rot.res.uniqueness}
reveals that $\np{\vvel,\vpres}\in\Xreys(\Omega)\times\Ys(\Omega)$ for all $q\in(1,2)$.
\end{proof}

Now we can combine the previous lemmas
to conclude the proof Theorem \ref{thm:Oseen.rot.res} by a final approximation argument.

\begin{proof}[Proof of Theorem \ref{thm:Oseen.rot.res}]
In the case $\Omega=\R^3$ the statement follows from Theorem \ref{thm:Oseen.rot.res.R3} above.
If $\Omega\subset\R^3$ is an exterior domain,
the uniqueness statement is a consequence of Lemma \ref{lem:Oseen.rot.res.uniqueness},
and estimate \eqref{est:Oseen.rot.res}
was shown in Lemma \ref{lem:Oseen.rot.res.estimates}.
For existence, let $g\in\LR{q}(\Omega)^3$ and 
consider a sequence $(g_j)\subset\CRci(\Omega)^3$
converging to $g$ in $\LR{q}(\Omega)$.
By Lemma \ref{lem:Oseen.rot.res.exist.smooth}
there exists a solution 
$\np{\vvel_j,\vpres_j}\in\Xreys(\Omega)\times\Ys(\Omega)$
to \eqref{sys:Oseen.rot.res} with $g=g_j$ for each $j\in\N$.
From Lemma \ref{lem:Oseen.rot.res.estimates}
we infer that $\np{\vvel_j,\vpres_j}$ 
is a Cauchy sequence in 
$\Xreys(\Omega)\times\Ys(\Omega)$.
Since this is a Banach space,
there exists a unique limit $\np{\vvel,\vpres}\in\Xreys(\Omega)\times\Ys(\Omega)$,
which is a solution to \eqref{sys:Oseen.rot.res}.
This completes the proof.
\end{proof}

\section{Existence of time-periodic solutions}
\label{sec:time-periodic}
Now we consider the linear and nonlinear 
time-periodic problems \eqref{sys:Oseen.rot.tp}
and \eqref{sys:NS.rot.tp},
and we prove the existence results from Theorem \ref{thm:Oseen.rot.tp}
and Theorem \ref{thm:NS.rot.tp}.
We begin with the well-posedness of the linear problem \eqref{sys:Oseen.rot.tp}.

\begin{proof}[Proof of Theorem \ref{thm:Oseen.rot.tp}]
Let $f\in\AR(\torus;\LR{q}(\Omega)^3)$,
and let $f_k\in\LR{q}(\Omega)^3$, $k\in\Z$, such that
\[
f(t,x)=\sum_{k\in\Z}f_k(x)\e^{i\perf kt}.
\]
By Theorem \ref{thm:Oseen.rot.res},
for each $k\in\Z$
there is a solution $\np{\uvel_k,\upres_k}\in\Xreyk(\Omega)\times\Ys(\Omega)$
to
\begin{equation}\label{sys:Stokes.rot.res.uk}
\begin{pdeq}
i\perfs k\uvel_k+\tay\rotterm{\uvel_k} 
- \Delta\uvel_k
-\rey\partial_1\uvel_k
+ \grad \upres_k
 &= f_k
&& \tin \Omega, \\
\Div\uvel_k&=0
&& \tin\Omega, \\
\uvel_k&=0
&& \ton \partial\Omega,
\end{pdeq}
\end{equation}
which satisfies
\[
\begin{aligned}
&\norm{i\perfs k\uvel_k+\tay\rotterm{\uvel_k}}_{q}
+\norm{\grad^2 \uvel_k}_{q}
+\norm{\grad \uvel_k}_{q}
\\
&\qquad\qquad
+\norm{\grad \uvel_k}_{3q/(3-q)}
+\norm{\uvel_k}_{3q/(3-2q)}
+\norm{\upres_k}_{3q/(3-q)}
\leq C\norm{f_k}_{q}.
\end{aligned}
\]
Since $\perf/\tay\in\Q$,
Proposition \ref{prop:linearcomb} implies that
condition \eqref{est:rey2.tp} can always be satisfied for some $\theta>0$.
Moreover, \eqref{est:rey2.tp} implies \eqref{est:rey2.res}
for all $s\in\perf\Z$,
so that the constant $C$ can be chosen independently of $k\in\Z$.
Therefore,
the series
\[
\uvel(t,x)=\sum_{k\in\Z}\uvel_k(x)\e^{i\perf kt},
\qquad
\upres(t,x)=\sum_{k\in\Z}\upres_k(x)\e^{i\perf kt},
\]
define a pair 
$\np{\uvel,\upres}\in\XzeroT(\torus\times\Omega)\times\YT(\torus\times\Omega)$,
which satisfies estimate \eqref{est:Oseen.rot.tp}
with the same constant $C$
and solves the time-periodic problem \eqref{sys:Oseen.rot.tp}.

To show the uniqueness statement, consider a solution 
$\np{\uvel,\upres}\in\XreyT(\torus\times\Omega)\times\YT(\torus\times\Omega)$
to \eqref{sys:Oseen.rot.tp} with $f=0$.
Then the Fourier coefficients 
$\np{\uvel_k,\upres_k}$, $k\in\Z$,
are elements of $\Xreyk(\Omega)\times\Ys(\Omega)$
and satisfy
\eqref{sys:Stokes.rot.res.uk} with $f_k=0$.
From Theorem \ref{thm:Oseen.rot.res}
we thus conclude $\np{\uvel_k,\upres_k}=\np{0,0}$
for all $k\in\Z$,
so that $\np{\uvel,\upres}=\np{0,0}$. 
This shows uniqueness of the solution and
completes the proof.
\end{proof}

Next we provide the proof of Theorem \ref{thm:NS.rot.tp}
on the existence of a solution to the nonlinear problem \eqref{sys:NS.rot.tp} 
for ``small'' data,
which is based on a fixed-point argument.

\begin{proof}[Proof of Theorem \ref{thm:NS.rot.tp}]
The proof largely follows that of 
\cite[Theorem 2.3]{EiterKyed_ViscousFlowAroundRigidBodyPerformingTPMotion_2021},
where existence of a solution to \eqref{sys:NS.rot.tp} was shown for $\perf=\tay$,
which is why we skip the details here.

Let $R>0$ such that $\partial\Omega\subset\ball_R$,
and let $\varphi\in\CRci(\R^3)$ 
such that $\varphi(x)=1$ for $\snorm{x}\leq R$, and $\varphi(x)=0$ for $\snorm{x}\geq 2R$. 
Define $\Uvel(x)\coloneqq\frac{1}{2}\rot\nb{\np{\rey\eone\wedge x - \tay\eone\snorm{x}^2}\varphi(x)}$.
Then $\np{\uvel,\upres}$ is a solution to \eqref{sys:NS.rot.tp}
if and only if $\np{\wvel,\wpres}\coloneqq\np{\uvel-\Uvel,\upres}$
is a solution to
\begin{equation}\label{sys:NS.lifted}
\begin{pdeq}
\partial_t \wvel + \tay\rotterm{\wvel} 
-\rey\partial_1\wvel
-\Delta \wvel + \grad \wpres 
&= f+\caln(\wvel) 
&& \tin \torus\times\Omega, \\
\Div\wvel&=0
&& \tin \torus\times\Omega, \\
\wvel&=0
&& \ton \torus\times\partial\Omega, \\
\lim_{\snorm{x}\to\infty} \wvel(t,x) &= 0
&& \tfor t\in \R,
\end{pdeq}
\end{equation}
where 
\[
\caln(\wvel)
\coloneqq-\tay\rotterm{\Uvel}+\rey\partial_1\Uvel+\Delta\Uvel-\np{\wvel+\Uvel}\cdot\grad\np{\wvel+\Uvel}.
\]
It now suffices to show existence of a solution $\np{\wvel,\wpres}$ to \eqref{sys:NS.lifted}.

On the space $\XreyT(\torus\times\Omega)$ we define the norm $\norm{\cdot}_{q,\rey,\tay}$ by
\[
\begin{aligned}
\norm{\zvel}_{q,\rey,\tay}
&\coloneqq\norm{\grad^2\zvel}_{\AR(\torus;\LR{q}(\Omega))}
+\norm{\pt\zvel+\tay\rotterm\zvel}_{\AR(\torus;\LR{q}(\Omega))}
+\rey\norm{\partial_1\zvel}_{\AR(\torus;\LR{q}(\Omega))}
\\
&\qquad
+\rey^{1/2}\norm{\zvel}_{\AR(\torus;\LR{2q/(2-q)}(\Omega))}
+\rey^{1/4}\norm{\grad\zvel}_{\AR(\torus;\LR{4q/(4-q)}(\Omega))}.
\end{aligned}
\]
As in \cite[Proof of Theorem 2.3]{EiterKyed_ViscousFlowAroundRigidBodyPerformingTPMotion_2021},
one then derives the estimates
\[
\begin{aligned}
\norm{\caln(\zvel)}_{\AR(\torus;\LR{q}(\Omega))}
&\leq C\bp{
\np{\rey+\tay}\np{1+\rey+\tay+\norm{\zvel}_{q,\rey,\tay}}
+\rey^{-\frac{3q-3}{q}}\norm{\zvel}_{q,\rey,\tay}^2
},
\\
\norm{\caln(\zvel_1)-\caln(\zvel_2)}_{\AR(\torus;\LR{q}(\Omega))}
&\leq C\bp{\rey+\tay
+\rey^{-\frac{3q-3}{q}}\np{
\norm{\zvel_1}_{q,\rey,\tay}+\norm{\zvel_2}_{q,\rey,\tay}
}}
\norm{\zvel_1-\zvel_2}_{q,\rey,\tay}
\end{aligned}
\]
for all $\zvel,\zvel_1,\zvel_2\in\XreyT(\torus\times\Omega)$ and $q\in[\frac{12}{11},\frac{6}{5}]$,
where $C$ is independent of $\rey$ and $\tay$.

Since we assume $\perf/\tay\in\Q$,
Theorem \ref{thm:Oseen.rot.tp}
provides the existence of a solution operator 
$\cals\colon\LR{q}(\torus\times\Omega)^3\to\XreyT(\torus\times\Omega)$
that maps a right-hand side $f\in\LR{q}(\torus\times\Omega)^3$
onto the velocity field $\uvel\in\XreyT(\torus\times\Omega)$ of a solution 
$\np{\uvel,\upres}$ to \eqref{sys:Oseen.rot.tp}.
We further introduce the set 
$A_\delta\coloneqq\setc{\zvel\in\XreyT(\torus\times\Omega)}{\norm{\zvel}_{q,\rey,\tay}\leq\delta}$.
Then $\np{\wvel,\wpres}$ is a solution to \eqref{sys:NS.lifted}
if $\wvel$ is a fixed point of the mapping
$\calf\colon A_\delta\to \XreyT(\torus\times\Omega)$,
$\zvel\mapsto\cals(f+\caln(\zvel))$.
Now let $\rey_0>0$ and $0<\tay\leq\tay_0\coloneqq\kappa\rey_0^\rho$
and assume that \eqref{est:rey2.nonlin} as well as 
$\norm{f}_{\AR(\torus;\LR{q}(\Omega))}<\varepsilon$ are satisfied.
In virtue of the previous estimates and Theorem \ref{thm:Oseen.rot.tp},
we then obtain 
a constant $C$, independent of $\rey$ and $\tay$, such that
\[
\begin{aligned}
\norm{\calf(\zvel)}_{q,\rey,\tay}
&\leq C\bp{
\varepsilon+\np{\rey+\kappa\rey^\rho}\np{1+\rey+\kappa\rey^\rho+\delta}
+\rey^{-\frac{3q-3}{q}}\delta^2
},
\\
\norm{\calf(\zvel_1)-\calf(\zvel_2)}_{q,\rey,\tay}
&\leq C\bp{\rey+\kappa\rey^\rho
+\rey^{-\frac{3q-3}{q}}\delta
}
\norm{\zvel_1-\zvel_2}_{q,\rey,\tay}
\end{aligned}
\]
for $C$ independent of $\rey$ and $\tay$ and for all $\zvel,\zvel_1,\zvel_2\in A_\delta$.
From these estimates one readily derives that $\calf\colon A_\delta\to A_\delta$
is a contractive self-mapping
if we choose $\varepsilon=\rey^2$, $\delta=\rey^\mu$ for some $\mu\in(\frac{3q-3}{q},\rho)$,
and $0<\rey\leq\rey_0$ sufficiently small.
The contraction mapping principle then yields the existence of a unique fixed point
$\wvel\in A_\delta$, 
which finally shows the existence of a solution $\np{\uvel,\upres}$ to \eqref{sys:NS.rot.tp}.
\end{proof}

\section{Construction of the counterexamples}
\label{sec:counterexample}
In this section we 
shall prove Theorem \ref{thm:counterexample.tp} and Theorem \ref{thm:counterexample}.
We first address the latter and
construct an explicit counterexample,
which is a modification of the sequence from the proof of
\cite[Theorem 3.1]{DeuringVarnhorn_OseenResolventEst_2010},
where the Oseen problem \eqref{sys:Oseen.res} was considered.

\begin{proof}[Proof of Theorem \ref{thm:counterexample}]
Fix $\rey>0$ and consider a sequence $(\sigma_n)\subset\R$ 
with $\frac{1}{n}\leq \sigma_n \leq \frac{2}{n}$.
Since we assume that $\alpha/\tay\not\in\Q$,
Corollary \ref{cor:linearcomb.odd}
enables us to choose $\sigma_n=\alpha k_n+\tay\ell_n$ for $k_n,\,\ell_n\in\Z$
with $\ell_n$ odd.
Moreover, we have that $s_n\coloneqq\alpha k_n$
satisfies $\snorm{s_n}\geq\frac{1}{n}$ for all $n\in\N$ sufficiently large.
Indeed, if $\snorm{s_n}<\frac{1}{n}$, 
then $0<\tay\ell_n<\frac{3}{n}$,
which is impossible for large $n\in\N$ and $\ell_n\in\Z$.

For simplicity, we denote the spatial Fourier transform of a function 
$\varphi\in\LR{2}(\R^3)$ by $\ft\varphi\coloneqq\FT_{\R^3}\nb{\varphi}$.
By Plancherel's theorem,
we can define a sequence $(G_n)\subset\LR{2}(\R^3)^3$
by 
\[
\ft G_n(\xi)\coloneqq n^{3/2}\charfct_{I_n}(\xi) \frac{(0,\xi_3,-\xi_2)}{\snorm{\xi'}},
\]
where $\xi'\coloneqq(\xi_2,\xi_3)$ and
\[
I_n\coloneqq \np{\tfrac{1}{2\rey n},\tfrac{1}{\rey n}}\times A_n,
\qquad
A_n\coloneqq\setcl{\xi'=(\xi_2,\xi_3)\in\R^2}
{\snorm{\xi'}<\tfrac{1}{n}, \, \xi_3>0}.
\]
We have $\snorm{A_n}=\frac{\pi}{2n^2}$ 
and $\snorm{I_n}=\frac{\pi}{4\rey n^3}$,
which yields
\begin{equation}
\label{eq:Gnj.norm}
\norm{G_n}_{2}^2
=\norm{\ft G_n}_{2}^2
=n^{3} \snorml{I_n}
=\tfrac{\pi}{4\rey}.
\end{equation}
Now set
\[
\ft\Vvel_n(\xi)\coloneqq\frac{1}{i\sigma_n+\snorm{\xi}^2-i\rey\xi_1}\ft G_n(\xi).
\]
Since $\sigma_n\neq 0$, 
this defines a sequence $(\Vvel_n)\subset\LR{2}(\R^3)^3$ 
by Plancherel's theorem.
With a similar argument,
one shows $\grad^2\Vvel_n\in\LR{2}(\R^3)$.
Additionally, $\Vvel_n$ satisfies
\begin{equation}\label{sys:Oseen.Vn}
i\sigma_n\Vvel_n
- \Delta \Vvel_n
-\rey\partial_1\Vvel_n
= G_n,
\quad
\Div\Vvel_n
=0 
\qquad \tin\R^3,
\end{equation}
where we used $\xi\cdot\ft G_n(\xi)=0$.
In particular, we also have $\partial_1\Vvel_n\in\LR{2}(\R^3)$,
and $\Vvel_n$ satisfies a classical Oseen resolvent problem 
(with pressure $\Vpres_n\equiv 0$).
As in our approach in Section \ref{sec:resprob.wholespace},
we now transform this solution into 
a solution to the resolvent problem \eqref{sys:Oseen.rot.res.R3}.
To this end, we recall the rotation matrix $\rotmatrix_\tay$ 
from \eqref{eq:Q.def},
and we let $\per_\tay=\frac{2\pi}{\tay}$ and $\torus=\R/\per_\tay\Z$ 
denote the associated time period and torus group.
Set
\[
\begin{aligned}
\Uvel_n(t,x)
&\coloneqq\Vvel_n(x)\e^{i\tay\ell_n  t},
&\qquad
F_n(t,x)
&\coloneqq G_n(x)\e^{i\tay\ell_n t},
\\
\uvel_n(t,x)
&\coloneqq\rotmatrix_\tay(t)^\transpose\Uvel_n(t,\rotmatrix_\tay(t)x),
&\qquad
f_n(t,x)
&\coloneqq\rotmatrix_\tay(t)^\transpose F_n(t,\rotmatrix_\tay(t)x),
\\
\vvel_n(x)
&\coloneqq
\int_\torus\uvel_n(t,x)\,\dt,
&\qquad
g_n(x)
&\coloneqq
\int_\torus f_n(t,x)\,\dt.
\end{aligned}
\]
First of all, we have 
$\partial_t\Uvel_n=i\tay\ell_n\Vvel_n=(s_n-i\alpha k_n)\Vvel_n$,
so that
$i\alpha k_n\Uvel_n+\pt\Uvel_n$,
$\grad^2\Uvel_n$,
$\partial_1\Uvel_n\in\LR{2}(\torus\times\R^3)$
and
\[
i\alpha k_n\Uvel_n
+\partial_t\Uvel_n
- \Delta \Uvel_n
-\rey\partial_1\Uvel_n
= F_n,
\quad
\Div\Uvel_n
=0 
\qquad \tin\torus\times\R^3.
\]
Mimicking the calculations from the proof of Theorem \ref{thm:Oseen.tp.rot.mod.R3},
we can further deduce that we have
$i\alpha k_n\uvel_n+\partial_t\uvel_n+\tay\rotterm{\uvel_n},\,
\grad^2\uvel_n,\,
\partial_1\uvel_n\in\LR{2}(\torus\times\R^3)$
and
\[
i\alpha k_n\uvel_n+\partial_t\uvel_n
+\tay\rotterm{\uvel_n}
- \Delta \uvel_n
-\rey\partial_1\uvel_n
= f_n,
\quad
\Div\uvel_n
=0 
\qquad \tin\torus\times\R^3.
\]
Computing the time means over $\torus$, we now deduce
$i\alpha k_n\vvel_n+\tay\rotterm{\vvel_n}$,
$\grad^2\vvel_n$,
$\partial_1\vvel_n\in\LR{2}(\R^3)$
and
\[
i\alpha k_n\vvel_n
+\tay\rotterm{\vvel_n}
- \Delta \vvel_n
-\rey\partial_1\vvel_n
= g_n,
\quad
\Div\vvel_n
=0 
\qquad \tin\R^3.
\]
Hence, $\np{\vvel,\vpres}=\np{\vvel_n,0}$ is a solution 
to the resolvent problem \eqref{sys:Oseen.rot.res.R3}
with $g=g_n$ and $s=s_n=\alpha k_n\in\alpha\Z$,
and it belongs to the asserted function class.
It thus remains to show inequality \eqref{est:counterexample}.
At first, let us address the right-hand side.
Using that $\FT_{\R^3}$ is an isomorphism $\LR{2}(\R^3)\to\LR{2}(\R^3)$
and commutes with the transformation induced by $\rotmatrix_\tay(t)$, we obtain
\[
\norm{g_n}_{2}^2
=\int_{\R^3}\!\snormL{\int_\torus 
\rotmatrix_\tay(t)^\transpose
\ft G_n(\rotmatrix_\tay(t)\xi)
\e^{i\tay\ell_nt}
\dt}^2\dxi
=n^{3}\int_{\R^3}\!\snormL{\int_\torus 
\charfct_{I_n}(\rotmatrix_\tay(t)\xi)
\e^{i\tay\ell_nt}
\dt}^2\dxi.
\]
To first compute the integral over $\torus$, we fix $\xi\in\R^3$ for the moment,
let $\eta=(\xi_1,\snorm{\xi'},0)$ and
choose $\varphi\in[0,\per_\tay)$ such that
$\rotmatrix_\tay(\varphi)\eta=\xi$.
Employing the equivalence
\[
\begin{aligned}
\rotmatrix_\tay(\psi)\eta\in I_n
&\iff \np{\eta_1,\snorm{\eta'},\psi}
\in\np{\tfrac{1}{2\rey n},\tfrac{1}{\rey n}}
\times(0,\tfrac{1}{n})
\times(0,\tfrac{\pi}{\tay})
\\
&\iff
\np{\xi_1,\snorm{\xi'},\psi}
\in\np{\tfrac{1}{2\rey n},\tfrac{1}{\rey n}}
\times(0,\tfrac{1}{n})
\times(0,\tfrac{\pi}{\tay}),
\end{aligned}
\]
we infer
\[
\begin{aligned}
\int_\torus\!\charfct_{I_n}(\rotmatrix_\tay(t)\xi) \e^{i\tay\ell_n t}\dt
&=\int_\torus \charfct_{I_n}(\rotmatrix_\tay(t+\varphi)\eta) \e^{i\tay\ell_n t}\dt
=\e^{-i\tay\ell_n\varphi}
\int_\torus\! \charfct_{I_n}(\rotmatrix_\tay(t)\eta) \e^{i\tay\ell_n t}\dt
\\
&=\e^{-i\tay\ell_n\varphi} 
\,\charfct_{\bp{\tfrac{1}{2\rey n},\tfrac{1}{\rey n}}}(\xi_1)
\charfct_{[0,\tfrac{1}{n})}(\snorm{\xi'})\,
\iper\int_0^{\tfrac{\pi}{\tay}}\e^{i\tay\ell_n t}\,\dt
\\
&=\frac{1}{2\pi i\ell_n}\e^{-i\tay\ell_n\varphi} 
\bp{\e^{i\pi\ell_n}-1}\,
\charfct_{\bp{\tfrac{1}{2\rey n},\tfrac{1}{\rey n}}}(\xi_1)\,
\charfct_{\bp{0,\tfrac{1}{n}}}(\snorm{\xi'}).
\end{aligned}
\]
Since $\ell_n$ is an odd number, we have $\e^{i\pi\ell_n}=-1$
and conclude
\begin{equation}\label{eq:meanvalue.counterex}
\snormL{\int_\torus \charfct_{I_n}(\rotmatrix_\tay(t)\xi) \e^{i\tay\ell_n t}\,\dt}^2
=\frac{1}{\pi^2 \ell_n^2}\,
\charfct_{\bp{\tfrac{1}{2\rey n},\tfrac{1}{\rey n}}}(\xi_1)\,
\charfct_{\bp{0,\tfrac{1}{n}}}(\snorm{\xi'}),
\end{equation}
which leads to
\begin{equation}\label{eq:gn.norm}
\norm{g_n}_{2}^2
=\frac{n^{3}}{\pi^2 \ell_n^2}\int_{\R^3}
\charfct_{\bp{\tfrac{1}{2\rey n},\tfrac{1}{\rey n}}}(\xi_1)\,
\charfct_{\bp{0,\tfrac{1}{n}}}(\snorm{\xi'})\,\dxi
=\frac{1}{2\rey\pi \ell_n^2}.
\end{equation}
Now let us turn to the left-hand side of \eqref{est:counterexample}.
Arguing as above, we obtain
\begin{equation}\label{eq:norm.resrotterm.counterex}
\norm{i\alpha k_n\vvel_n+\tay\rotterm{\vvel_n}}_{2}^2
=s_n^2\!\int_{\R^3}\!\snormL{\int_\torus 
\rotmatrix_\tay(t)^\transpose
\ft\Vvel_n(\rotmatrix_\tay(t)\xi)
\e^{i\tay\ell_nt}
\dt}^2\dxi.
\end{equation}
Again, we focus on the integral over $\torus$ at first.
From the definition of $\ft\Vvel_n$ and
the invariance properties $\snorm{\rotmatrix_\tay(t)\xi}=\snorm{\xi}$
and $\rotmatrix_\tay(t)\eone=\eone$, 
we deduce the identity
\[
\begin{aligned}
\snormL{\int_\torus\rotmatrix_\tay(t)^\transpose
\ft\Vvel_n(\rotmatrix_\tay(t)\xi)\e^{i\tay\ell_n t}\,\dt}^2
&=\snormL{\int_\torus \frac{1}{is_n+\snorm{\xi}^2-i\rey\xi_1}
\rotmatrix_\tay(t)^\transpose
\ft G_n(\rotmatrix_\tay(t)\xi)\e^{i\tay\ell_n t}
\,\dt}^2
\\
&=\frac{n^{3}}{\snorml{is_n+\snorm{\xi}^2-i\rey\xi_1}^2}
\,\snormL{\int_\torus\charfct_{I_n}(\rotmatrix_\tay(t)\xi)
\e^{i\tay\ell_n t}\,\dt}^2
\\
&=\frac{n^3}{\pi^2 \ell_n^2}\,\frac{1}{\snorml{is_n+\snorm{\xi}^2-i\rey\xi_1}^2}\,
\charfct_{\bp{\tfrac{1}{2\rey n},\tfrac{1}{\rey n}}}(\xi_1)\,
\charfct_{\bp{0,\tfrac{1}{n}}}(\snorm{\xi'}),
\end{aligned}
\]
where we invoked \eqref{eq:meanvalue.counterex} for the last equality.
For $\frac{1}{2\rey n}<\xi_1<\frac{1}{\rey n}$ and $\snorm{\xi'}<\tfrac{1}{n}$
we have 
$s_n-\rey\xi_1>\frac{1}{n}-\rey\xi_1\geq0$,
and we can estimate
\[
\begin{aligned}
\frac{1}{\snorml{\snorm{\xi}^2+i(s_n-\rey\xi_1)}^2}
&
\geq\frac{1}{(\xi_1^2+\frac{1}{n^2})^2+(\frac{1}{n}-\rey\xi_1)^2}
\geq\frac{1}{2\xi_1^4+\frac{2}{n^4}+(\frac{1}{n}-\rey\xi_1)^2}\\
&\geq\frac{1}{\bp{\frac{\sqrt{2n^4\xi_1^4+2}}{n^2}+\frac{1}{n}-\rey\xi_1}^2},
\end{aligned}
\]
where the relation $a^2+b^2\leq (a+b)^2 \leq2(a^2+b^2)$ for $a,b\in[0,\infty)$
was employed in the last two estimates.
Returning to identity \eqref{eq:norm.resrotterm.counterex},
we have thus shown
\[
\begin{aligned}
\norm{i\alpha k_n\vvel_n+\tay\rotterm{\vvel_n}}_{2}^2
&\geq\frac{s_n^2 n^3}{\pi^2\ell_n^2}\!
\int_{\frac{1}{2\rey n}}^{\frac{1}{\rey n}}\int_{\snorm{\xi'}<\frac{1}{n}}
\frac{1}{\bp{\frac{\sqrt{2n^4\xi_1^4+2}}{n^2}+\frac{1}{n}-\rey\xi_1}^2}
\,\dxi'\dxi_1
\\
&
=\frac{s_n^2 n}{\pi\ell_n^2}\!
\int_{\frac{1}{2\rey n}}^{\frac{1}{\rey n}}
\frac{1}{\bp{\frac{\sqrt{2n^4\xi_1^4+2}}{n^2}+\frac{1}{n}-\rey\xi_1}^2}
\,\dxi_1.
\end{aligned}
\]
With
the transformation $\xi_1=\tfrac{\theta}{\rey n}$
we further obtain
\[
\begin{aligned}
&\norm{i\alpha k_n\vvel_n+\tay\rotterm{\vvel_n}}_{2}^2
=\frac{s_n^2 n^2}{\rey\pi\ell_n^2}
\int_{\frac{1}{2}}^{1} \!
\frac{1}{\Bp{\frac{1}{n}\sqrt{2+\frac{2\theta^4}{\rey^4}}+1-\theta}^2}
\,\dtheta
\\
&\quad
\geq\frac{s_n^2 n^2}{\rey\pi\ell_n^2}
\!\int_{\frac{1}{2}}^{1} \!
\frac{1}{\Bp{\frac{1}{n}\sqrt{2+\frac{2}{\rey^4}}+1-\theta}^2}
\,\dtheta
=\frac{s_n^2 n^2}{\rey\pi\ell_n^2}
\Bp{\frac{n}{\sqrt{2+\frac{2}{\rey^4}}}
-\frac{1}{\frac{1}{n}\sqrt{2+\frac{2}{\rey^4}}+\tfrac{1}{2}}}\\
&\quad
\geq\frac{s_n^2 n^2}{\rey\pi\ell_n^2}
\Bp{\frac{n}{\sqrt{2+\frac{2}{\rey^4}}}-2}
\end{aligned}
\]
If we choose $n\in\N$ large enough, that is,  
such that $n\geq 4\sqrt{2+\frac{2}{\rey^4}}$,
then
\[
\norm{i\alpha k_n\vvel_n+\tay\rotterm{\vvel_n}}_{2}^2
\geq\frac{s_n^2 n^2}{\rey\pi\ell_n^2}
\frac{n}{2\sqrt{2+\frac{2}{\rey^4}}}
=\frac{s_n^2 n^3}{2\rey\pi\ell_n^2\sqrt{2+\frac{2}{\rey^4}}}.
\]
With $\snorm{s_n}\geq\tfrac{1}{n}$ and \eqref{eq:gn.norm}, this yields
\[
\norm{i\alpha k_n\vvel_n+\tay\rotterm{\vvel_n}}_{2}^2
\geq\frac{n}{\sqrt{2+\frac{2}{\rey^4}}}\norm{g_n}_{2}^2,
\]
which corresponds to \eqref{est:counterexample} with $C=(2+2\rey^{-4})^{-1/4}$.
\end{proof}

To conclude Theorem \ref{thm:counterexample.tp},
we construct a suitable right-hand side $f$ as a Fourier series 
defined by means of the sequence $(g_n)$,
and we show that the corresponding solutions $\np{\vvel_n,\vpres_n}$ 
cannot constitute a Fourier series that is a solution to \eqref{sys:Oseen.rot.tp}
with the desired properties.

\begin{proof}[Proof of Theorem \ref{thm:counterexample.tp}]
Let us focus on case~\ref{it:counterexample.tp.i}~at first.
Consider the sequences $(s_n)\subset\perf\Z$ and $(g_n)\subset\LR{2}(\R^3)^3$ 
from Theorem \ref{thm:counterexample}
and the corresponding sequence of velocity fields $\np{\vvel_n}$. 
We may assume that $s_n\neq s_m$ if $n \neq m$ 
and, by a renormalization, that $\norm{g_n}_{2}=n^{-3/2}$.
Then we can define $f\in \AR(\torus;\LR{2}(\R^3)^3)$ by
\[
f(t,x)=\sum_{k\in\Z} f_k(x)\e^{i\perf k t},
\qquad
f_k=
\begin{cases}
g_n &\tif \perf k= s_n \text{ for some }n\in\N, \\
0 &\tif \perf k\neq s_n \text{ for all }n\in\N.
\end{cases}
\]
Now assume that there exists a solution 
$\np{\uvel,\upres}\in\LRloc{1}(\torus\times\R^3)^{3+1}$ 
to \eqref{sys:Oseen.rot.tp} 
such that $\pt\uvel+\tay\rotterm{\uvel},\,
\grad^2\uvel,\,
\partial_1\uvel
\in\AR(\torus;\LR{2}(\R^3)^3)$.
Let $\np{\uvel_k,\upres_k}\coloneqq\np{\FT_\torus\nb{\uvel}(k),\FT_\torus\nb{\upres}(k)}$, 
$k\in\Z$, 
be the Fourier coefficients of $\np{\uvel,\upres}$.
Then
\[
i\perf k\uvel_k
+\tay\rotterm{\uvel_k}
- \Delta \uvel_k
+ \grad \upres_k
= f_k,
\quad
\Div\uvel_k
=0 
\qquad \tin\R^3
\]
and $i\perf k\uvel_k+\tay\rotterm{\uvel_k},\,
\grad^2\uvel_k,\,
\partial_1\uvel_k
\in\LR{2}(\R^3)$.
The uniqueness statement from Theorem \ref{thm:Oseen.rot.res.R3}
now implies
that
$i\perf\uvel_k+\tay\rotterm{\uvel_k}=0$ if $\perf k\neq s_n$ for all $n\in\N$,
and that 
\[
i\perf\uvel_k+\tay\rotterm{\uvel_k}=is_n\vvel_n+\tay\rotterm{\vvel_n}
\]
if $\perf k= s_n$ for some $n\in\N$.
In virtue of inequality \eqref{est:counterexample}, 
we thus conclude
\[
\begin{aligned}
\sum_{k\in\Z}\norm{i\perf k\uvel_k+\tay\rotterm{\uvel_k}}_{2}
&=\sum_{n=1}^\infty\norm{is_n\vvel_n+\tay\rotterm{\vvel_n}}_{2}
\\
&\geq C \sum_{n=1}^\infty n^{1/2} \norm{g_n}_{2}=\infty.
\end{aligned}
\]
Hence, the left-hand side is not summable
and $\pt\uvel+\tay\rotterm\uvel\not\in\AR(\torus;\LR{2}(\R^3)^3)$,
which contradicts the assumption 
and completes the proof of \ref{it:counterexample.tp.i}

For \ref{it:counterexample.tp.ii}~we proceed in the same way but renormalize $g_n$ such that
$\norm{g_n}_{2}=n^{-1}$.
By Plancherel's theorem we then obtain $f\in\LR{2}(\torus\times\R^3)^3$
and
\[
\begin{aligned}
\sum_{k\in\Z}\norm{i\perf k\uvel_k+\tay\rotterm{\uvel_k}}_{2}^2
&=\sum_{n=1}^\infty\norm{is_n\vvel_n+\tay\rotterm{\vvel_n}}_{2}^2
\\
&\geq C \sum_{n=1}^\infty n \norm{g_n}_{2}^2=\infty.
\end{aligned}
\]
This shows $\pt\uvel+\tay\rotterm\uvel\not\in\LR{2}(\torus\times\R^3)^3$
and completes the proof.
\end{proof}

\end{document}